\documentclass[11pt, reqno]{amsart}
\usepackage{amsmath}
\usepackage{amsthm}
\usepackage{amssymb}
\usepackage{amscd}
\usepackage{amsbsy}
\usepackage{epsfig}
\usepackage{graphicx}
\usepackage{psfrag}
\usepackage{bbm}
\usepackage[pagebackref,bookmarks=false]{hyperref}
\usepackage[UKenglish]{babel}
\usepackage[UKenglish]{isodate}
\usepackage{cleveref}
\usepackage{stackrel}

\theoremstyle{plain}
\newtheorem{theorem}{Theorem}[section]

\newtheorem{remark}[theorem]{Remark}
\newtheorem{proposition}[theorem]{Proposition}
\newtheorem{lemma}[theorem]{Lemma}
\newtheorem{corollary}[theorem]{Corollary}

\theoremstyle{definition}
\newtheorem{definition}[theorem]{Definition}

\newtheorem*{claim*}{Claim}

\newtheorem{maintheorem}{Theorem}

\newtheorem*{lemma*}{Lemma}

\theoremstyle{definition}

\def\R{\ensuremath{\mathbb R}}
\def\N{\ensuremath{\mathbb N}}

\def\e{{\ensuremath{\rm e}}}

\def\M{\ensuremath{\mathcal M}}

\def\P{\ensuremath{\mathcal P}}

\def\E{\mathbb E}

\def\dist{\ensuremath{\text{dist}}}

\def\eps{\varepsilon}

\def\1{ \mathbbm{1}}

\numberwithin{equation}{section}

\begin{document}

\author[S.~Baker]{Simon Baker}
\address{Mathematical Sciences\\
Loughborough University\\
Loughborough\\
Leicestershire\\
LE11 3TU\\
UK} 
\email{simonbaker412@gmail.com}
\urladdr{https://simonbakermaths.wordpress.com/}    

\author[M.~Todd]{Mike Todd}
\address{Mathematical Institute\\
University of St Andrews\\
North Haugh\\
St Andrews\\
KY16 9SS\\
Scotland} 
\email{m.todd@st-andrews.ac.uk}
\urladdr{https://mtoddm.github.io}

\title{Pair correlation statistics for dynamical systems}
\date{\today}



\maketitle

\begin{abstract}
	We study the pair correlation statistics of orbits generated by maps on the interval. We show that under suitable mixing and multifractal assumptions, the pair correlation statistics of an orbit will almost surely exhibit the same asymptotic behaviour as a suitable sequence of i.i.d.\ random variables. We will also show that, under suitable hypotheses, the pair correlation statistics defined by two orbits will almost surely exhibit the same behaviour as two suitable sequences of i.i.d.\ random variables. Specific dynamical systems to which our results apply to include Gibbs-Markov maps and the Gauss map.  We also give an example of a slowly mixing system for which the pair correlation statistics of an orbit almost surely behave distinctly to an i.i.d.\ sequence.
\end{abstract}

\section{Introduction}
Understanding the distribution of sequences of real numbers modulo one is a well-studied and important problem in mathematics \cite{Bugeaudbook,KuiNie}. One of the first topics one encounters in this field is the notion of uniform distribution: a sequence of real numbers $(x_{n})_{n}$ is \emph{uniformly distributed modulo one} if for any $0\leq a<b<1$ we have
$$\lim_{n\to\infty}\frac{\# \{1\leq i\leq n: x_{i} \mod 1\in [a,b]\}}{n}=b-a.$$
It is a consequence of the strong law of large numbers that if $(X_n)_n$ is a sequence of i.i.d.\ random variables distributed according to the uniform measure on $[0,1]$, then $(X_{n})_n$ will be uniformly distributed modulo one almost surely. This statement, and the concept of uniform distribution modulo one, can be generalised to other measures rather than just the uniform measure. It is a well studied problem to determine whether a sequence, often of some dynamical or number theoretic origins, is uniformly distributed with respect to some measure. The famous Birkhoff Ergodic Theorem \cite[Theorem 1.14]{Waltersbook} shows that for ergodic, measure-preserving transformations the orbit of a typical point will behave like a suitable sequence of i.i.d.\ random variables and be uniformly distributed with respect to the underlying measure. 

A more fine scale understanding of the distribution of a sequence is provided by the pair correlation statistics. Namely, given a sequence of real numbers $(x_n)_n$, $\beta\in (0,2)$ and $s>0$, we would like to understand the asymptotic behaviour of $$\#\left\{1\leq i\neq j\leq n:\|x_{i}-x_{j}\|\leq \frac{s}{n^{\beta}}\right\}$$ as $n\to\infty$. Here and throughout we let $\|x\|=\min\{|x-n|:n\in\mathbb{Z}\}$ for $x\in \mathbb{R}$.  It is known that if $\beta\in(0,2)$ and $(X_n)_n$ is a sequence of i.i.d.\ random variables distributed according to the uniform measure on $[0,1],$ then for any $s>0$ we have $$\lim_{n\to\infty}\frac{1}{n^{2-\beta}}\#\left\{0\le i\neq  j< n:\|X_i-X_j\|\le \frac s{n^{\beta}}\right\}=2s$$ almost surely. Just as in the case of uniform distribution, there is an analogue of this statement for more general measures than the uniform measure. Of particular interest is the case where $\beta=1$. In this case a sequence of real numbers $(x_n)_n$ is said to have Poissonian pair correlations if $$\lim_{n\to\infty}\frac{1}{n}\#\left\{0\le i\neq  j< n:\|x_i-x_j\|\le \frac s{n}\right\}=2s$$ for all $s>0$. It is known that if $(x_n)_n$ has Poissonian pair correlations then it is uniformly distributed modulo one \cite{AisLacPau18,GreLar17}. This result has been generalised to higher dimensions and sequences of points on manifolds in \cite{Mar20}. The opposite implication however does not hold, the sequence $(\alpha n)_{n}$ does not have Poissonian pair correlations for any $\alpha\in \mathbb{R},$ but is uniformly distributed modulo one whenever $\alpha\in \mathbb{R}\setminus \mathbb{Q}$. Thus the pair correlation statistics can be viewed as providing more detailed information about the distribution of a sequence than being uniformly distributed, and consequently they give a more sophisticated measure for how random a sequence is. 

With the above in mind, suppose we were given a sequence of some dynamical or number theoretic origins: it is natural to ask whether it behaves like a sequence of i.i.d.\ random variables at the level of pair correlation statistics. This is a particularly natural question to ask in the dynamical context given that the discussion above demonstrates that pair correlation statistics can provide more detailed information than the conclusion of the Birkhoff Ergodic Theorem. Moreover, it fits into an ongoing and active area of research in dynamics that is concerned with determining to what extent orbits behave like a sequence of i.i.d random variables, for example Central Limit Theorems/Stable Laws and Large Deviations results stemming from works such as \cite{You98, AarDen01}  and \cite{MelNic08, ReyYou08} respectively, and more recently functional limit theorems such as \cite{MelZwe15, FreFreTod25}.  We will also later mention results involving recurrence properties, which are more directly relevant to this paper, on hitting time statistics, shrinking target problems and extreme value theory. This being said, very little is known about the pair correlation statistics of orbits generated by dynamical systems, with the exception being two works by Rudnick and Zaharescu \cite{RudZah99, RudZah02}. It follows from \cite{RudZah99,RudZah02} that for $b\in \mathbb{N}_{\geq 2},$ if $T_{b}:\mathbb{R}/\mathbb{Z}\to \mathbb{R}/\mathbb{Z}$ is given by $x\to bx \mod 1$, then for Lebesgue almost every $x$ the sequence $(T_{b}^{n}x)_n$ has Poissonian pair correlations and the local spacing measures coincide with those of the Poisson model (see \cite[Section 1.1]{RudZah02} for further details). Part of the motivation for this paper is to begin to extend these results beyond such specialised dynamical systems. To the best of the authors' knowledge, this is the first paper to consider this problem for general families of dynamical systems on the interval. In contrast, understanding the pair correlation statistics of a sequence of some number theoretic origin is an active area of research. The case when $(x_n)_n=(\alpha q_n)_n$ where $(q_n)$ is a sequence of natural numbers and $\alpha\in \mathbb{R}$ has received significant attention \cite{ AisLarLew17,AisLarTec19,BloChoGalWal18,NaiPol07,RudZah99, RudZah02, Wal}. This is in part due to the connection between $(\alpha q_n)$ having Poissonian pair correlations for Lebesgue almost every $\alpha$ and the additive energy of the sequence $(q_n)_n$. There are also interesting connections between this problem and the Berry-Tabor conjecture from quantum mechanics \cite{RudSar}. Another well studied problem is to determine those values of $\theta>0$ for which $(n^{\theta})_n$ has Poissonian pair correlations. We refer the reader to \cite{LutSouTec25} for more on this problem, and for a proof that $(n^{\theta})$ has Poissonian pair correlations for all $\theta\in (0,14/41)$. We also refer the reader to \cite{ElkMcM04} and \cite{BazMarVin} where a deep connection between the distributional properties of $(n^{1/2})_n$ and homogeneous dynamics is exploited. In \cite{AisBak} powers of real numbers were studied, and it was shown that for Lebesgue almost every $\alpha>1$ the sequence $(\alpha^{n})_n$ has Poissonian pair correlations.    

In this paper we study the pair correlation statistics of orbits generated by maps $T:X\to X$ where $X\subset I=[0,1]$. More formally, given $s>0, \beta>0$, $n\in \N$ and $x\in X,$ we are interested in the almost sure asymptotic behaviour of
\begin{equation*}
	\#\left\{0\le i\neq  j< n:|T^ix-T^jx|\le \frac s{n^{\beta}}\right\}
	\label{eq:mainquestion}.
\end{equation*} Proceeding via analogy with the i.i.d case, one would expect that if $\mu$ is a $T$-invariant probability measure that is suitably mixing, then for an interval of $\beta,$ for $\mu$-almost every  $x$ we should have 
\begin{equation}
	\label{eq:expectedoneorbitasymp}
\lim_{n\to\infty}\frac{\#\left\{0\le i\neq  j< n:|T^ix-T^jx|\le \frac s{n^{\beta}}\right\}}{n^{2}\int \mu(B(z,\frac{s}{n^{\beta}}))\, d\mu(z)}= 1
\end{equation} for all $s>0.$ In Theorem \ref{mainthm:1orb} we verify this prediction for a family of interval maps and measures.  Notice that the divisor here is $n^{2}$ multiplied by the probability that $|x-y|\le s/n^{\beta}$ for $x, y$ chosen independently according to $\mu$, which if $\mu$ equals Lebesgue is $2sn^{2-\beta}$. 

We will also consider the pair correlations statistics generated by two orbits. That is for $s>0, \beta>0$, $n\in \N$ and $x,y\in X,$ we are interested in the almost sure asymptotic behaviour of
\begin{equation*}
	\#\left\{0\le i , j< n:|T^ix- T^jy|\le \frac s{n^{\beta}}\right\}
	\label{eq:mainquestiontwoorbit}.
\end{equation*}
Again we would expect that if $\mu$ is a $T$-invariant probability measure that is suitably mixing, then for an interval of $\beta$, for $(\mu\times \mu)$-almost every $(x,y)$ we should have 
\begin{equation}
	\label{eq:twoorbitexpectedasym}
\lim_{n\to\infty}\frac{\#\left\{0\le i,  j< n:d(T^ix, T^jy)\le \frac s{n^{\beta}}\right\}}{n^{2}\int \mu(B(z,\frac{s}{n^{\beta}}))\, d\mu(z)}=1
\end{equation} for all $s>0$. Theorem \ref{mainthm:2orb} similarly verifies this prediction for a family of interval maps and measures. 

Our main applications are to interval maps with good mixing properties, for example Gibbs-Markov interval maps and their absolutely continuous (with respect to Lebesgue measure) invariant probability measure (\emph{acip}), but as we show, our results also apply more widely to invariant probability measures with non-trivial multifractal behaviour. We also consider cases where the conclusion of our theory does not hold, i.e.\ the pair correlation statistics of a $\mu$-typical point do not coincide with those of a suitable i.i.d.\ sequence. Such examples are known to exists for ergodic transformations, e.g. by the discussion above irrational rotations of the circle have this property. This could however be attributed to the failure of the mixing property for these dynamical systems. It is natural to ask whether a quantitative rate of mixing would be sufficient to guarantee that the orbit of $\mu$-almost every $x$ behaves like an i.i.d.\ sequence from the perspective of pair correlation statistics. A natural place to look for examples, as in the case of the least distance problem (described below) in \cite{RouTod24},  is interval maps with subexponential mixing. By considering a suitable class of Manneville-Pomeau maps, we show that even a polynomial rate of mixing is insufficient to guarantee this property (see Section \ref{subsec:counterexample}).

We finish this opening discussion by highlighting a connection between pair correlation statistics and the notion of recurrence in dynamical systems. The foundational result in the study of recurrence is Poincar\'{e}'s recurrence theorem \cite[Theorem 1.4]{Waltersbook} which states that if $(X,\mathcal{A},\mu)$ is a probability space and $T:X\to X$ is a measure preserving transformation, then for any $A\in \mathcal{A}$ for $\mu$ almost every $x\in A$ there exist infinitely many $n\in \N$ such that $T^{n}x\in A$. Under fairly modest assumptions, it can be shown that if $X$ is equipped with a metric $d$, then Poincar\'{e}'s recurrence theorem implies that $\liminf_{n\to\infty}d(T^{n}x,x)=0$ for $\mu$-almost every $x$. It is natural to ask whether a more quantitative statement holds. Following earlier work of Boshernitzan \cite{Bos93} this problem has recently received significant interest. We refer the reader to \cite{AllBakBar25,BakFar21,BakKoi24,BarLiaRou19,Bos93,FreFreTod10,HolTod25,KirKunPer23, HolKirKunPer24, LeLiSiVe25,Sau09} and the references therein for more on this quantitative approach to recurrence. One approach to this problem is to study the smallest distance between all pairs of points chosen from an orbit, and between pairs of points chosen from two distinct orbits. More formally, this means understanding the quantities
$$M_{n}(x):=\min\{d(T^ix,T^jx): 0\leq i\neq j<n\}$$ and $$M_n(x, y) :=\min\left\{d(T^ix, T^jy): 0\le i, j<n\right\}$$ for $x,y\in X$ and $n\in \N$. The study of $M_{n}(\cdot)$ and $M_{n}(\cdot,\cdot)$ is related to longest substring problems in probability \cite{ArrWat85}, which have applications in the comparison of DNA strands \cite{ReScWa05,Wat95}. The study of pair correlation statistics can be viewed as a counting counterpart to this problem, where we count how many times a rate of recurrence is observed within an orbit. In \cite{BarLiaRou19}, the authors considered how the quantity
$M_n(x, y)$
scales with $n$ for $(\mu\times \mu)$-almost every $(x, y)$ for $\mu$ a $T$-invariant probability measure. For sufficiently nice dynamical systems and measures it was shown that 
$$\lim_{n\to\infty}\frac{-\log M_n(x, y)}{\log n} = \frac2{C_{\mu}}$$ for $(\mu\times\mu)$-almost every $(x, y),$
where $C_\mu$ is the correlation dimension (defined below). This result is significant as it tells us that the optimal range of $\beta$ for which we could expect the convergence in \eqref{eq:expectedoneorbitasymp} and \eqref{eq:twoorbitexpectedasym} to hold is $\beta\in (0,2/C_{\mu})$.  Moreover,  \cite{BarLiaRou19} introduces techniques which are useful for the problems addressed in the current paper.  In the case of a single orbit, where there are extra challenges with dependence which don't exist in the two-orbit case, a similar result was proved in \cite{Zha24}. A more precise analysis of the two-orbit case was done in \cite{KirKunPerTod25}, where near-optimal results were proved for liminf and limsup subsets defined in terms of the behaviour of $M_{n}(\cdot,\cdot)$. We also note that the case of random iteration was considered in \cite{GouRouSta24}. Recently for $(\alpha q_n)_n$, where $(q_n)_n$ is a sequence of integers and $\alpha\in \R$, the analogue of this minimal gaps problem was considered in \cite{Rud18, Reg23}. There the results relied on properties of the additive energy of the sequence.

\subsection{Definitions and main theorems}

We first define properties associated with our measures.
Given a Borel probability measure $\mu$ supported on $[0,1]$, we define the \emph{upper correlation dimension} of $\mu$ to be
$$\overline{C}_\mu:= \limsup_{r\searrow 0} \frac{\log \int\mu(B(x, r))~d\mu(x)}{\log r}.$$
 Similarly we define $\underline{C}_\mu$ the \emph{lower correlation dimension of $\mu$} via the liminf.  We generalise this to two Borel probability measures $\mu_1, \mu_2$, by defining  
 the \emph{upper correlation dimension of $(\mu_1, \mu_2)$} to be 
 $$\overline{C}_{\mu_1, \mu_2}:= \limsup_{r\searrow 0} \frac{\log \int\mu_1(B(x, r))~d\mu_2(x)}{\log r}$$
and similarly $\underline{C}_{\mu_1, \mu_2}$ using the liminf (note that by Fubini's theorem, switching $\mu_1$ and $\mu_2$ does not change these quantities).

  Given $s>0$ we say that a Borel probability measure $\mu$ is \emph{$s$-Ahlfors regular} if there exists
  $C>0$ such that
  \[
    \frac1Cr^s\le \mu(B(x, r))\le Cr^s 
  \]
  for all $x\in \mathrm{supp}(\mu)$ and $0<r<1$. Particular examples of $1$-Ahlfors regular measures are absolutely continuous invariant probability measures (acips) with density uniformly bounded away from zero and infinity. Acips are known to exist, and to satisfy the mixing assumptions appearing in our theorems, when $T$ is a Gibbs-Markov map (see Definition \ref{def:Gibbs-Markov}). Gibbs-Markov maps include piecewise smooth full branched interval maps with bounded distortion. Some of the measures we will consider will not be $s$-Ahlfors regular for any $s>0$. In this case a relevant scaling quantity is the Frostman dimension. We define the \emph{Frostman dimension} of a Borel probability measure $\mu$ supported on $[0,1]$ to be the supremum of those $s>0$ for which there is $C>0$ such that 
  $$ \mu(B(x, r))\le Cr^{s}$$ for all $x\in [0,1]$ and $0<r<1$.  We denote the Frostman dimension of such a measure by $F_{\mu}$.

\begin{definition}
We say that two Borel probability measures $\mu_1$ and $\mu_2$  have \emph{continuous mean scaling} if for any  $s, \beta,\eps>0$, there exists $\delta_0>0$ such that $0\le \delta\le \delta_0$ implies
$$ \left|1-\frac{\int \mu_1(B(x,(s\pm \delta)/n^\beta) d\mu_2}{\int \mu_1(B(x,s/n^\beta) d\mu_2}\right| < \eps$$ for all $n\in\mathbb{N}$. If $\mu_1=\mu_2$ we say that \emph{$\mu_1$ has continuous mean scaling}.
\end{definition}

We next define the mixing properties of our dynamical systems.  
  
  \begin{definition}
  Let $(X, T, \mu)$ denote a measure preserving system where $X\subset [0,1],$ and let $(\mathcal{C}^1, \|\cdot\|_{1})$ and $(\mathcal{C}^2, \|\cdot\|_{2})$ be Banach spaces of observables on $[0,1]$. If there exists
  $C, \theta>0$ such that for all $\psi\in \mathcal{C}^1$ and
  $\phi\in \mathcal{C}^2$,
  \[
    \biggl|\int \psi\cdot \phi\circ T^n \, d\mu- \int \psi \,
    d\mu \int \phi \, d\mu \biggr| \le
    C\|\psi\|_{1}\|\phi\|_{2} \e^{-\theta n},
  \]
  then we say that $(X, T, \mu)$ has \emph{exponential mixing for
    $\mathcal{C}^1$ against $\mathcal{C}^2$}.
\end{definition}

We will usually assume that $\mathcal{C}^1$ are the functions of bounded variation $BV$ with the usual norm $\|\psi\|_{BV} = |\psi|_{Var} + \|\psi\|_{L^1}$, and that $\mathcal{C}^2=L^\infty$ with the sup norm.

\begin{definition}
  Let $(X, T, \mu)$ denote a measure preserving system where $X\subset [0,1]$. If there exist $C', \theta'>0$ such that whenever $\1_{1},\1_{2},\1_{3},\1_{4}$ are indicator functions on intervals and $0
  \leq a < b \leq c,$ then
  \begin{align*}
    & \biggl | \int \1_1\cdot \1_2\circ T^a\cdot \1_3\circ
      T^b\cdot 1_4\circ T^c \, d\mu - \int \1_1\cdot
      \1_2\circ T^a \, d\mu \int \1_3\cdot
      \1_4\circ T^{c-b} \, d\mu \biggr|\\
    &\hspace{8cm} \le C' \e^{-\theta'
      (b-a)},
  \end{align*}
  then we say that $(X, T, \mu)$ has \emph{exponential 4-mixing
    for intervals}.
\end{definition}
Note that in the definition of exponential $4$-mixing for intervals that if we set $\1_1=\1_4=\mathrm{Id}$ then we obtain a similar statement but for products of two indicator functions on intervals. This statement coincides with what one would obtain from assuming exponential mixing for $BV$ against $L^{\infty}.$ The exponential $4$-mixing for intervals property is known to hold for Gibbs--Markov interval maps, see \cite[Lemma~4.16]{Zha24}.

Our first main theorem is as follows; note that a version not assuming the existence of $C_{\mu_1, \mu_2}$ is given later in Theorem~\ref{thm:2orbgen}.

\begin{maintheorem}
Suppose that $(X, T_1, \mu_1)$ and $(X, T_2, \mu_2)$ satisfy the following properties:
\begin{itemize}
\item $(X, T_i, \mu_i)$ have exponential mixing for
    $BV$ against $L^\infty$ for $i=1, 2$. 
\item  $C_{\mu_1}, C_{\mu_2}$ and $C_{\mu_1, \mu_2}$ exist, belong to $(0, \infty)$, and satisfy $C_{\text{max}}>0$ for $C_{\text{max}}:= \max_{i=1, 2}\left\{2C_{\mu_1, \mu_2}-C_{\mu_i}\right\}$.
\item  $\mu_1$  and $\mu_2$ have continuous mean scaling.
\end{itemize}
  Then for
  $\beta\in (0, 2/C_{\text{max}})$, for  $(\mu_1\times\mu_2)$-a.e.\ $(x, y)$, for all $s>0$, 
$$\lim_{n\to \infty}  \frac{ \#\left\{0\le i, j<n:|T_1^ix- T_2^jy|\le \frac s{n^{\beta}}\right\}}{n^2\int\mu_1\left(B\left(z,\frac s{n^{\beta}}\right)\right)~d\mu_2(z)}= 1.$$
\label{mainthm:2orb}
\end{maintheorem}

An example of an application of Theorem \ref{mainthm:2orb} is to the following type of interval map, which we give a particular version of.

\begin{definition}
	\label{def:Gibbs-Markov}
Let $\tilde X\subset [0,1]$ be a set that can be written as $\tilde X=\cup_{P\in\P} P$ where $\P$ is an at most countable collection of intervals.  Suppose $\tilde T:\tilde X\to I$ is a function satisfying the following properties:
\begin{itemize}
\item for each $P\in \P$, $\tilde T(P)\cap \tilde X$ is a union of elements of $\P$;
\item $\tilde T$ is $C^{1}$ on each $P\in \P$ and there exist $\iota, C_\iota>0$ such that for any $P\in \P$, if $x, y\in P$ then $|D\tilde T(x)- D\tilde T(y)|\le C_\iota|x-y|^\iota$;

\item there exists $\lambda>1$ such that $|D\tilde T(x)|\ge\lambda$ for all $x\in \tilde X$;
\item we have bounded distortion: there exists $C>0$ such that if $x, y\in P\in \P$ then $\left|\frac{D\tilde T(x)}{D\tilde T(y)}-1\right|\le C|x-y|$;
\item there exists $b_0>0$ such that $Diam(\tilde T(P))\ge b_0$ for all $P\in \P$.
\end{itemize}
Define the \emph{attractor} to be
$$X = X_{\tilde T}:= \left\{x\in \tilde{X}: \tilde T^ix\in \cup_{P\in \P}P\text{ for all } i\ge 0\right\}.$$ We denote the restriction of $\tilde{T}$ to $X$ by $T:X\to X$ and call $T$ a \emph{Gibbs-Markov map}.
\end{definition}

Note that we will sometimes abuse notation and discuss $T$ acting outside of the attractor $X$. If $X_{\tilde T} =I$, then these maps have a unique acip, with density $\rho$ in BV and bounded away from 0 and $\infty,$ and which satisfies the mixing conditions of our main theorems: in this case we will write $T=\tilde T$.  Basic examples of these maps are $x\mapsto kx \mod 1$ for $k\in \N$, see also Farey maps and L\"uroth maps (eg \cite{KesMunStr12}). 
At the same time as considering these maps we will consider the \emph{Gauss map} $x\mapsto 1/x\mod 1$ with acip given by the Gauss measure.  We emphasise that the Gauss map behaves like a Gibbs-Markov map, see for example \cite[Section 2.6.3]{BruDemTod18}, despite there being a point $x$ where $|DT(x)|=1$ (i.e., for $x=1$). 
The proof of the following is given in Section~\ref{sec:apps}.

\begin{corollary}
Suppose that $(I, T_i)$ are Gibbs-Markov maps or the Gauss map, with acips and densities denoted by $\mu_i$ and $\rho_i$ for $i=1, 2$, respectively. 
 Then for  $\beta\in (0, 2)$,
$$\frac1{n^{2-\beta}\int(\rho_1\rho_2)(t)~dt}  \#\left\{0\le i, j< n:|T_1^ix- T_2^jy|\le \frac s{n^{\beta}}\right\}\to 2s,$$
for $\mu_1\times\mu_2$ a.e. $(x, y)$.
\label{cor:2orbacip}
\end{corollary}

 So in the 
 special case when $\mu_1=\mu_2=$ Lebesgue, for  $\beta\in (0, 2)$ $$\frac1{n^{2-\beta}}  \#\left\{0\le i, j< n:|T_1^ix- T_2^jy|\le \frac s{n^{\beta}}\right\}\to 2s.$$
 
 A Gibbs-Markov map where $\P$ is a finite set, its elements have disjoint closures and $\tilde T(P)= I$ for all $P\in \P$ is called a \emph{cookie cutter}, a version of which is in \cite{Ran89}.  Note that here $X_{\tilde T} \neq I$ and 
  the Hausdorff dimension $h$ of $X_{\tilde T}$ is strictly less than 1.
 If $\phi:\tilde X\to \R$ is H\"older on each $P\in \P$ then there is a unique equilibrium state $\mu=\mu_\phi$.  Note one example of this is $\phi=-h\log|D\tilde T|$, in which case $\mu_\phi$ is the natural `geometric measure' and is $h$-Ahlfors regular. It is the analogue to the acips mentioned above in this case. The system $(X, T, \mu)$ then satisfies the conditions of our main theorems as we show in Section~\ref{sec:apps}, and hence we have the conclusion of Theorem~\ref{mainthm:2orb} for this family of maps and measures. This conclusion will also hold more generally for equilibrium states (see Section~\ref{sec:apps} for the definition).

Our second main theorem considers the case of a single orbit. This introduces significantly more dependence and has a much more involved proof.  In this case we require an extra condition (as well as exponential 4-mixing for intervals).  Given a measure preserving transformation $(X,T,\mu)$ where $X\subset [0,1]$, let $$A_r (n) := \{\, x : |T^n x-x| < r \,\}.$$
 We say that $(X,T,\mu)$ satisfies the \emph{early return property} if there exist $C,s>0$ such that  
\begin{equation}
	\mu \left(A_r(n)\right) \leq C r^{s}
    \label{eq:shortreturnestimate}
\end{equation}
for all $r>0$ and $n\in \N$. We let $D_{\mu}$ denote the supremum of those $s>0$ for which \eqref{eq:shortreturnestimate} holds for some $C>0$ and call $D_{\mu}$ the \emph{early return exponent}. We give examples of systems satisfying the early return property in Section~\ref{sec:apps}.

\begin{maintheorem}
Assume that $(X,T,\mu)$ satisfies the following properties:
\begin{itemize}
	\item $(X, T, \mu)$ has exponential mixing for
	$BV$ against $L^\infty$ and exponential $4$-mixing for intervals.
	\item $(X, T, \mu)$ has the early return property with exponent $D_{\mu}$.
	\item The correlation dimension of $\mu$ exists and is strictly positive. We denote it by $C_{\mu}$.
	\item $\mu$ has continuous mean scaling.
\end{itemize}
Then for $\beta\in (0,2/C_{\mu})$ satisfying the following inequalities:
\begin{itemize}
	\item $\beta(C_{\mu}-D_{\mu})<1,$
	\item $\beta(C_{\mu}-F_{\mu})<1,$
\end{itemize}
for $\mu$ almost every $x,$ for all $s>0$ we have
$$\lim_{n\to \infty}  \frac{ \#\left\{0\le i\neq  j<n:|T^ix- T^jx|\le \frac s{n^{\beta}}\right\}}{n^2\int\mu\left(B\left(z,\frac s{n^{\beta}}\right)\right)~d\mu(z)}= 1.$$
\label{mainthm:1orb}
\end{maintheorem}

As in the two-orbit case, if we have an acip with a density in BV, as for example in the case of Gibbs-Markov or Gauss maps, we have a result which is simpler to state. 

\begin{corollary}
Suppose that $(I, T)$ is a Gibbs-Markov map or the Gauss map, with acip $\mu$ and density $\rho$.  Then for all   $\beta\in (0, 2)$, for $\mu$-a.e.\ $x$ and all $s>0$,
$$\frac1{n^{2-\beta}\int\rho(t)^2~dt}  \#\left\{0\le i\neq  j< n:|T^ix- T^jx|\le \frac s{n^{\beta}}\right\}\to 2s.$$
\label{cor:1orbacip}
\end{corollary}
So again in  the special case when $\mu=$ Lebesgue, for  $\beta\in (0, 2)$ $$\frac1{n^{2-\beta}}  \#\left\{0\le i\neq  j< n:|T^ix- T^jx|\le \frac s{n^{\beta}}\right\}\to 2s,$$
for Lebesgue-a.e.\ $x$.  We also have applications to cookie cutters and their associated equilibrium states, see Section~\ref{sec:apps}.

\begin{remark}
We note that there are other possible shrinking sequences that we could use to define our pair correlation statistics that we can cover using our techniques, but the polynomial decay is a helpful property here, and the literature is mostly concerned with sequences of the form $(s/n^\beta)_n$, so we restrict ourselves to these sequences here.
\end{remark}

We conclude our introduction by noting that there are yet finer ways in which one could examine the randomness of an orbit.  For instance, one could follow the programme of Peres and Weiss, see \cite{Wei20, AlvBecMer23} and study the point processes corresponding to orbits.

\textit{Structure of the paper:}  The rest of the paper is organised as follows. In Section \ref{sec:2orb} we prove Theorem \ref{thm:2orbgen} which implies Theorem \ref{mainthm:2orb}. Our proof of Theorem \ref{mainthm:1orb} is split across Sections \ref{sec:one orbit first part} and \ref{sec:one orbit second moment argument}. In Section \ref{sec:one orbit first part} we introduce a conditioning argument that will underpin our proof of Theorem \ref{mainthm:1orb} and obtain some useful expectation bounds. Finally in Section \ref{sec:one orbit second moment argument} we complete our proof by obtaining suitable second moment estimates. In Section \ref{sec:apps} we explain why Corollaries \ref{cor:2orbacip} and \ref{cor:1orbacip} follow from our theorems and include some additional examples. We also include in this section the details for the aforementioned example of an interval map and invariant measure with a polynomial rate of mixing for which the orbit of almost every point does not have the expected pair correlation statistics.

\textit{Notation:} For sequences $(a_n)_n$ and $(b_n)_n$ in $[0, \infty)$ we  write $a_n\lesssim b_n$ if there is $C>0$ such that $a_n\le Cb_n$ for all $n$.  The notation $a_n\gtrsim b_n$ is defined similarly.  If $a_n\lesssim b_n$ and $a_n \gtrsim b_n$, we write $a_n \asymp b_n$. We write $a_n\sim b_n$ if $\frac{a_n}{b_n} \to 1$.  For $(c_n)_n$ with  $c_n\in \R$ we also write $c_n= \mathcal{O}(b_n)$ if $|c_n|\lesssim b_n$.  Finally, if $(a_n)_n$ in $[0, \infty)$ satisfies $\lim_{n\to\infty} a_{n}=0$, then we will sometimes write $a_{n}=o(1)$.

\textit{Acknowledgements.} SB was partially funded by his EPSRC
New Investigators Award (EP/W003880/1). MT was partially funded by his EPSRC grant UKRI1120. MT also thanks the University of Loughborough for their hospitality.

\section{Proof for the two orbit case}

\label{sec:2orb}

Theorem~\ref{mainthm:2orb} follows immediately from the following more general statement. We note that after setting $\gamma$ to be any positive constant less than $2-\beta C_{\mu_1, \mu_2}$ so that \eqref{eq:r_ncond} is satisfied, to satisfy \eqref{eq:slowgrow} for some $\alpha\in (0,2\gamma)$ for $i=1$ and $i=2$, it is sufficient to satisfy that $\beta(2C_{\mu_1, \mu_2}-C_{\mu_1})$ and $\beta(2C_{\mu_1, \mu_2}-C_{\mu_2})$ are less than $2$. This is precisely the second assumption in Theorem~\ref{mainthm:2orb}.

\begin{theorem}
Suppose that  $(X, T_i, \mu_i)$ have exponential mixing for  $BV$ against $L^\infty$ for $i=1, 2$ and $\mu_1$  and $\mu_2$ have continuous mean scaling.  Let $\beta>0$.
For $\gamma\in (0, 2)$, suppose that for all $s>0$,
\begin{equation}
n^2\int\mu_1\left(B\left(x,s/n^\beta\right)\right)~d\mu_2(x) \gtrsim n^\gamma,
\label{eq:r_ncond}
\end{equation}
where the implied constant may depend on $s$.
Moreover, suppose that there is $\alpha\in (0, 2\gamma)$ such that
\begin{equation}
n^2\int\mu_i\left(B\left(x,s/n^\beta\right)\right)~d\mu_i(x) \lesssim n^\alpha
\label{eq:slowgrow}
\end{equation}
for $i=1$ and $i= 2$, where again the implied constant may depend on $s$.

Then for $\mu_{1}\times \mu_{2}$ a.e.\ $(x, y),$ for all $s>0$ we have 
$$\lim_{n\to \infty}  \frac{ \#\left\{0\le i, j<n:|T_1^ix- T_2^jy|\le \frac s{n^\beta}\right\}}{n^2\int\mu_1\left(B\left(z,s/n^\beta\right)\right)~d\mu_2(z)}= 1.$$
\label{thm:2orbgen}
\end{theorem}

The following lemma clarifies how \eqref{eq:r_ncond}, \eqref{eq:slowgrow} and $C_{\mu_1, \mu_2}$ are related.

\begin{lemma} \begin{enumerate}
\item[(a)]
If $\overline{C}_{\mu_1, \mu_2}<\infty$  then \eqref{eq:r_ncond} is satisfied for $s, \beta, \gamma$ whenever $\beta< \frac{2-\gamma}{\overline{C}_{\mu_1, \mu_2}}$.

\item[(b)]
If $\underline{C}_{\mu_1, \mu_2}>0$ and  \eqref{eq:r_ncond} is satisfied for $s, \beta, \gamma$, then $\beta< \frac{2-\gamma}{\underline{C}_{\mu_1, \mu_2}}$.
\item[(c)]
If $\underline{C}_{\mu}>0$ then  \eqref{eq:slowgrow} is satisfied for $s, \beta, \alpha$ whenever  $\beta> \frac{2-\alpha}{\underline{C}_{\mu}}$.
\end{enumerate}
\label{lem:Cdim}
\end{lemma}

\begin{proof}
 For (a) it is sufficient to set $s=1$.  Then 
for $\eps>0$ and $n$ sufficiently large,
$$\frac{\log \int\mu_1(B(x, 1/n^\beta))~d\mu_2(x)}{-\beta\log n}< \overline{C}_{\mu_1, \mu_2}+\eps.$$ 
So 
$$n^2\int\mu_1(B(x, 1/n^\beta))~d\mu_2(x) > n^{2-\beta\left(\overline{C}_{\mu_1, \mu_2}+\eps\right)}.$$
Therefore if $\beta<\frac{2-\gamma}{\overline{C}_{\mu_1, \mu_2}+\eps}$, we are finished, so we conclude by noting that $\eps>0$ was arbitrary.

For (b), for $\eps>0$ and sufficiently large $n$,  $\frac{\log \int\mu_1(B(x, s/n^\beta))~d\mu_2(x)}{\log(s/n^\beta)}\ge \underline{C}_{\mu_1, \mu_2}-\eps$ and \eqref{eq:r_ncond} imply 
$$n^\gamma \lesssim n^2\int\mu_1(B(x, s/n^\beta))~d\mu_2(x) \le n^2\left(\frac s{n^\beta}\right)^{\underline{C}_{\mu_1, \mu_2}-\eps},$$
so $n^{-\beta}\gtrsim n^{\frac{\gamma -2}{\underline{C}_{\mu_1, \mu_2}-\eps}}$ for all $\eps>0$, so $\beta< \frac{2-\gamma}{\underline{C}_{\mu_1, \mu_2}}$  as required.

Part (c) follows similarly.
\end{proof}

We will require the following straightforward lemmas, see Lemmas 3.2 and 3.3 of \cite{KirKunPerTod25}.

\begin{lemma}
  Suppose that $(X,\mu_1)$, $(X,\mu_2)$ are Borel probability spaces with $X\subset \R^n$. Then there exists $K>0$ such that
  if $r>0$ is sufficiently small then for $y\in X$,
  \[
    \mu_i(B(y, r))\le K \biggl(\int\mu_i(B(x,
      r)) \, d\mu_i(x) \biggr)^{\frac12}\quad \text{for } i=1,2.
  \]
  Hence also
    \[
    \int\mu_i(B(x, r))^2 \, d\mu_j(x) \le K \biggl(\int\mu_i(B(x,
      r)) \, d\mu_i(x)\biggr)^{\frac12}\int\mu_i(B(x,
    r)) \, d\mu_j(x).
  \]
  \label{lem:tight}
  for $(i,j)=(1,2), (2,1)$.
\end{lemma}

\begin{lemma}
  Let $(X,\mu)$ denote a Borel probability space where $X\subset
  \R$. For any $r>0$, $\psi_r \colon X \to \R$ given by
  $y\mapsto \mu(B(y,r))$ is a function of bounded variation with total variation bounded
  above by 2.
  \label{lem:BV}
\end{lemma}

\begin{remark}
We observe that the first part of Lemma~\ref{lem:tight} implies that if $C_{\mu_1}, C_{\mu_2}$ and $C_{\mu_1, \mu_2}$ exist and lie in $(0, \infty)$, as in Theorem~\ref{mainthm:2orb},  then $C_{\text{max}}\ge 0$. Recall that $C_{\text{max}}=\max_{i=1, 2}\left\{2C_{\mu_1, \mu_2}-C_{\mu_i}\right\}.$  \end{remark}

It will be useful in our proofs to use the following function.

\begin{definition} Let $\ell:\N\to \R^+$ be a function satisfying the conditions:
\begin{itemize}
\item for any $\delta>0$, $\frac{\ell(n)}{n^\delta} \to 0$ as $n\to\infty$;
\item for any $p>0$ and $\delta>0$, $n^p \e^{-\delta\ell(n)} \to 0$ as $n\to\infty$.
\end{itemize}
\label{def:ell}
\end{definition}

So for concreteness, we can choose $\ell(n) = \log n \log\log n$.

 \begin{proof}[Proof of Theorem~\ref{thm:2orbgen}] 
  Suppose, without loss of
  generality, that the exponential mixing constants are the same
  for both systems. 
  
Let us fix $\beta>0,$ $\gamma\in (0,2)$ and $\alpha\in (0,2\gamma)$ so that our assumptions are satisfied. For $s>0$, let $r_n=s/n^\beta$.  Then for $x, y\in X$,
  define
  \[
    S_n(x, y):= \sum_{i, j\in [0, n)} \1_{B(T_2^jy,
      r_{n})}(T_1^ix).
  \]
Our proof depends on showing that the following second moment bound is satisfied: Let $\hat \gamma:= \gamma-\frac\alpha2>0$ (see \eqref{eq:slowgrow}),  then for all $s>0$ there exists $C>0$ such that
\begin{equation}
\E \left(\frac{S_n}{\E(S_n)}- 1\right)^2 <\frac{C}{n^{\frac{\hat \gamma}2}}.
\label{eq:mainest}
\end{equation}
for all $n\in \N$. We will now explain why our result follows from \eqref{eq:mainest} before returning to its proof. Using Markov's inequality, we deduce
  \begin{align*}
(\mu_1\times \mu_2)\left(\left|\frac{S_n}{\E(S_n)}- 1\right|\geq \frac1{n^{\frac{\hat \gamma}8}}\right)&=  (\mu_1\times \mu_2)\left(\left(\frac{ S_n}{\E(S_n)}- 1\right)^2\geq \frac1{n^{\frac{\hat \gamma}4}}\right)\\
&\le n^{\frac{\hat \gamma}4}\E \left(\frac{ S_n}{\E(S_n)}- 1\right)^2\le \frac{C}{n^{\frac{\hat \gamma}4}}.
  \end{align*}

So taking a subsequence $(n^K)_n$ where $K>4/{\hat \gamma}$,  we see that 
$$(\mu_1\times \mu_2)\left(\left|\frac{S_{n^K}}{\E(S_{n^K})}- 1\right|\geq \frac1{n^{K\frac{\hat \gamma}8}}\right)
\le \frac{C}{n^{\frac{K{\hat \gamma}}4}},$$
which is summable. 
 Hence, by the Borel-Cantelli lemma for $\mu_{1}\times \mu_{2}$ a.e. $(x, y)$, 
$$\frac{S_{n^K}(x,y)}{\E(S_{n^K})}\to 1,$$
i.e., 
\begin{equation}
\lim_{n\to \infty}\frac{ \#\left\{0\le i, j< n^K:|T_1^i x- T_2^jy|\le \frac s{n^{\beta K}}\right\}}{n^{2K}\int\mu_1\left(B\left(z,s/ n^{K\beta}\right)\right)~d\mu_2(z)}   = 1.
\label{eq:polyspeed}
\end{equation}

  This establishes our desired convergence along a subsequence for a fixed choice of $s$. It remains to establish this convergence along the integers and for arbitrary $s>0$. 

Recalling that our sequence of radii are $(s/n^\beta)_n$ where $s>0$ was arbitrary, it follows from the above and the fact that a countable union of sets of measure zero has measure zero, that for $\mu_1\times\mu_2$ a.e. $(x, y)$, for any $s\in \mathbb{Q}_{>0}$ we have 
\begin{equation}
	\label{eq:rational limit2orb}
\lim_{n\to\infty}\frac{\#\left\{0\leq i, j\leq n^{K}:|T_1^ix-T_2^jy|\leq \frac{s}{n^{K\beta}}\right\}}{n^{2K}\int \mu(B(z,s/n^{K\beta}))~d\mu(z)}=1.
\end{equation} 
We emphasise that here we are using the fact that $\hat\gamma$ and $K$ are independent of $s$.

We will now explain why for $\mu_1\times\mu_2$ a.e $(x, y)$ for any $s>0$ we have 
$$\lim_{n\to\infty}\frac{\#\left\{0\leq i, j\leq n:|T_1^ix-T_2^jy|\leq \frac s{n^\beta}\right\}}{n^{2}\int \mu(B(z,s/n^\beta))~d\mu(z)}=1.$$ Let us fix $(x, y)$ belonging to the full measure set for which  \eqref{eq:rational limit2orb} holds for every $s\in \mathbb{Q}_{>0}$. For $n\in \N$ let $m\in \N$ be such that $m^{K}\leq n<(m+1)^{K},$ then for this choice of $x, y$ we have 
\begin{align*}
&\limsup_{n\to\infty}\frac{\#\{0\leq i,  j\leq n:|T_1^ix-T_2^jy|\leq \frac s{n^\beta}\}}{n^{2}\int \mu(B(z,s/n^\beta))~d\mu(z)}\\
&\leq \limsup_{n\to\infty}\frac{\#\{0\leq i,  j\leq (m+1)^K: |T_1^{i}x-T_2^{j}y|\leq \frac s{n^\beta}\}}{m^{2K}\int \mu(B(z,s/(m+1)^{K\beta}))~d\mu(z)}\\
&= \limsup_{n\to\infty}\frac{\#\{0\leq i, j\leq (m+1)^K:|T_1^{i}x-T_2^{j}y|\leq \frac s{(m+1)^{K\beta}}\frac{(m+1)^{K\beta}}{n^\beta}\}}{m^{2K}\int \mu(B(z,s/(m+1)^{K\beta}))~d\mu(z)}\\
&\leq\inf_{s'\in \mathbb{Q}:s'>s} \limsup_{n\to\infty}\frac{\#\{0\leq i,  j\leq (m+1)^K:|T_1^{i}x-T_2^{j}y|\leq \frac{s'}{(m+1)^{K\beta}}\}}{m^{2K}\int \mu(B(z,s/(m+1)^{K\beta}))~d\mu(z)}\\
&\leq\inf_{s'\in \mathbb{Q}:s'>s} \limsup_{n\to\infty}\frac{\#\{0\leq i,  j\leq (m+1)^K:|T_1^{i}x-T_2^{j}y|\leq \frac{s'}{(m+1)^{K\beta}}\}}{(m+1)^{2K}\int \mu(B(z,s'/(m+1)^{K\beta}))~d\mu(z)}\\
&\hspace{4cm}\times \frac{\int \mu(B(z,s'/(m+1)^{K\beta}))~d\mu(z)}{\int \mu(B(z,s/(m+1)^{K\beta}))~d\mu(z)}\times \frac{(m+1)^{2K}}{m^{2K}}\\
&=1 \times \inf_{s'\in\mathbb{Q}:s'>s}\limsup_{n\to\infty}\frac{\int \mu(B(z,s'/(m+1)^{K\beta}))~d\mu(z)}{\int \mu(B(z,s/(m+1)^{K\beta}))~d\mu(z)}=1. 
\end{align*}

In the final line we have used our assumption that $\mu_1$ and $\mu_2$  have continuous mean scaling.  It can similarly be shown that 
\begin{align*}
	&\liminf_{n\to\infty}\frac{\#\{0\leq i,  j\leq n:|T^ix-T^jy|\leq \frac s{n^\beta}\}}{n^{2}\int \mu(B(z,s/n^\beta))~d\mu(z)}\\
	&\geq \liminf_{n\to\infty}\frac{\#\{0\leq i\neq j\leq m^K:|T^{i}x-T^{j}y|\leq \frac s{n^\beta}\}}{(m+1)^{2K}\int \mu(B(z,s/m^{K\beta}))~d\mu(z)}\\
	&\geq \liminf_{n\to\infty}\frac{\#\{0\leq i\neq  j\leq m^K:|T^{i}x-T^{j}y|\leq \frac{s}{m^{K\beta}}\frac{m^{K\beta}}{n^\beta}\}}{(m+1)^{2K}\int \mu(B(z,s/m^{K\beta}))~d\mu(z)}\\
	&\geq\sup_{s'\in \mathbb{Q}:s'<s} \liminf_{n\to\infty}\frac{\#\{0\leq i\neq  j\leq m^K:|T^{i}x-T^{j}y|\leq \frac{s'}{m^{K\beta}}\}}{(m+1)^{2K}\int \mu(B(z,s/m^{K\beta}))~d\mu(z)}\\
	&\geq\sup_{s'\in \mathbb{Q}:s'<s} \liminf_{n\to\infty}\frac{\#\{0\leq i\neq  j\leq m^K:|T^{i}x-T^{j}y|\leq \frac{s'}{m^{K\beta}}\}}{m^{2K}\int \mu(B(z,s/m^{K\beta}))~d\mu(z)}\\
	&\hspace{5cm}\times \frac{\int \mu(B(z,s'/m^{K\beta}))~d\mu(z)}{\int \mu(B(z,s/m^{K\beta}))~d\mu(z)}\times \frac{m^{2K}}{(m+1)^{2K}}\\
	&=1 \times \sup_{s'\in\mathbb{Q}:s'<s}\liminf_{n\to\infty}\frac{\int \mu(B(z,s'/m^{K\beta}))~d\mu(z)}{\int \mu(B(z,s/m^{K\beta}))~d\mu(z)}=1. 
\end{align*} Where in the final line we used that $\mu_1$ and $\mu_2$  have continuous mean scaling. Thus $(x, y)$ satisfies 
$$\lim_{n\to\infty}\frac{\#\{0\leq i\neq  j\leq n:|T^ix-T^jy|\leq \frac s{n^\beta}\}}{n^{2}\int \mu(B(z,s/n^\beta))~d\mu(z)}=1$$ for any $s>0$. This completes our proof up to verifying \eqref{eq:mainest}.

Let us fix $s>0$ and define $r_n = s/n^\beta$. We now set out to prove \eqref{eq:mainest} for this choice of $s$. 
Using $T_1$-invariance of $\mu_1$ and $T_2$-invariance of $\mu_2$ we first compute
\begin{equation}
	\label{eq:Expectation lower bound two orbit}
\E(S_n)= n^{2}\int\mu_1(B(y, r_{n})) \, d\mu_2(y)\gtrsim n^\gamma. 
\end{equation}
  
Next, note that the left hand side of \eqref{eq:mainest} is
equal to $\frac{\E(S_n^2)-\E(S_n)^2}{\E(S_n)^2}$. Since we already have an expression for $\E(S_n)$ we proceed to
estimate $\E(S_n^2)$. We have
  \[
    \E(S_n^2) = \sum_{i_1, i_2, j_1, j_2 \in [0, {n})}
    \int\int\1_{B(T_2^{j_1}y, r_{n})} (T_1^{i_1}x)\cdot
    \1_{B(T_2^{j_2}y, r_{n})} (T_1^{i_2}x) \,
    d\mu_1(x)d\mu_2(y).
  \]
  In what follows we will often assume that $i_1\le i_2$ and $j_1\le j_2$ since for the three other combinations of inequalities, the upcoming arguments will give the same bound.  Similarly to the proof of  \cite[Theorem 3]{BarLiaRou19}, we will
  split the summation above according to the subcases corresponding to the following inequalities: $|i_2-i_1|\le \ell(n)$,
  $|i_2-i_1|> \ell(n)$, $|j_2-j_1|\le \ell(n)$, $|j_2-j_1|> \ell(n)$, where we recall $\ell$ from Definition~\ref{def:ell}.  
    
  For $i_2-i_1> \ell(n)$ and $j_2-j_1> \ell(n)$, i.e.\ \emph{the totally separated case} (the cases $i_1-i_2> \ell(n)$ and/or $j_1-j_2> \ell(n)$ follow similarly), the following holds 
      \begin{align*}
    & \sum_{\substack{i_1, i_2\in [0, {n})\\ i_2-i_1> \ell(n)}} \
    \sum_{\substack{j_1, j_2\in [0, {n})\\  j_2-j_1> \ell(n)}}
    \int\int\1_{B(T_2^{j_1}y, r_{n})} (T_1^{i_1}x)
    \1_{B(T_2^{j_2}y, r_{n})}(T_1^{i_2}x) \,
    d\mu_1(x)d\mu_2(y)\\ 
    &= \sum_{\substack{i_1, i_2\in [0, {n})\\ i_2-i_1> \ell(n)}} \
    \sum_{\substack{j_1, j_2\in [0, {n})\\  j_2-j_1> \ell(n)}}
    \int\int\1_{B(T_2^{j_1}y, r_{n})} (x)
    \1_{B(T_2^{j_2}y, r_{n})}(T_1^{i_2-i_1}x) \,
    d\mu_1(x)d\mu_2(y)\\ 
    &  \le \sum_{\substack{i_1, i_2\in [0,{n})\\ i_2-i_1>
    \ell(n)}} \ \sum_{\substack{j_1, j_2\in [0, {n})\\   j_2-j_1>
    \ell(n)}}  \biggl( \int\mu_1 (B(y, r_{n}))
    \mu_1 (B(T_2^{j_2-j_1}y, r_{n})) \, d\mu_2(y)  \\
    & \hspace{10cm} + \mathcal{O}\left( \e^{-\theta\ell(n)}\right)\biggr)\\
    &   \le \sum_{\substack{i_1, i_2\in [0, {n})\\
    i_2-i_1> \ell(n)}} \ \sum_{\substack{j_1, j_2\in [0, {n})\\
    j_2-j_1> \ell(n)}}  \biggl( \biggl(\int\mu_1 (B(y, r_{n}))
    \, d\mu_2(y) \biggr)^2 + \mathcal{O}\left( \e^{-\theta\ell(n)}\right)\biggr)\\
    & \leq \frac{n^{4}}4\biggl( \biggl(\int\mu_1 (B(y,
      r_{n})) \, d\mu_2(y)\biggr)^2+\mathcal{O}\left( \e^{-\theta\ell(n)}\right)\biggr)\\
      & = \frac{\E( S_n)^2}4+  \mathcal{O}\left(n^{4}\e^{-\theta\ell(n)}\right) ,
  \end{align*}
  where the first two inequalities follow from our exponential mixing assumptions, in the second inequality we have also used Lemma~\ref{lem:BV}. Applying the same argument to the other three cases then gives us a bound of 
  \begin{align}
  	\label{eq:First part of variance bound}
  	&\sum_{\substack{i_1, i_2\in [0, {n})\\ |i_2-i_1|> \ell(n)}} \
  	\sum_{\substack{j_1, j_2\in [0, {n})\\  |j_2-j_1|> \ell(n)}}
  	\int\int\1_{B(T_2^{j_1}y, r_{n})} (T_1^{i_1}x)
  	\1_{B(T_2^{j_2}y, r_{n})}(T_1^{i_2}x) \,
  	d\mu_1(x)d\mu_2(y)\nonumber\\
  	&\hspace{6cm}\leq \E( S_n)^2 + \mathcal{O}\left(n^{4} \e^{-\theta\ell(n)}\right).
  \end{align}
  We emphasise that the second term on the right hand side of \eqref{eq:First part of variance bound} decays faster than any polynomial by our choice of function $\ell$ (recall Definition~\ref{def:ell}).

  For $0\leq i_2-i_1\le  \ell(n)$ and $0\leq j_2-j_1\le \ell(n)$, i.e.\ \emph{the totally non-separated case}, we have  \begin{align*}
    & \sum_{\substack{i_1, i_2\in [0, {n})\\ 0\le i_2-i_1\le \ell(n)}}
    \ \sum_{\substack{j_1, j_2\in [0, {n})\\  0\le j_2-j_1\le \ell(n)}} 
    \int\int\1_{B(T_2^{j_1}y, r_{n})} (T_1^{i_1}x) \cdot
    \1_{B(T_2^{j_2}y, r_{n})} (T_1^{i_2}x) \,
    d\mu_1(x)d\mu_2(y)\\
    &\hspace{2.5cm} \le  \sum_{\substack{i_1, i_2\in [0, {n})\\
    0\le i_2-i_1\le \ell(n)}} \ \sum_{\substack{j_1, j_2\in [0,
    {n})\\
    0\le j_2-j_1\le \ell(n)}} \int \mu_1 (B(T_2^{j_1}y, r_{n}))
    \, d\mu_2(y)\\
    &\hspace{2.5cm} \le \quad \ell(n)^2 n^{2}\int \mu_1 (B(y,
      r_{n})) \, d\mu_2(y)=  \ell(n)^2 \E( S_n). 
  \end{align*} Again the cases where $0\leq i_1-i_2\le  \ell(n)$ and/or $0\leq j_1-j_2\le  \ell(n)$ follow similarly, and together yield
  \begin{align}
  	\label{eq:Second part of variance bound}
  	&\sum_{\substack{i_1, i_2\in [0, {n})\\ 0\le |i_2-i_1|\le \ell(n)}}
  	\ \sum_{\substack{j_1, j_2\in [0, {n})\\  0\le |j_2-j_1|\le \ell(n)}} 
  	\int\int\1_{B(T_2^{j_1}y, r_{n})} (T_1^{i_1}x) \cdot
  	\1_{B(T_2^{j_2}y, r_{n})} (T_1^{i_2}x) \,
  	d\mu_1(x)d\mu_2(y)\nonumber\\
  	&\hspace{8cm} \lesssim \ell(n)^{2}\E(S_n).
  \end{align}

  For $i_2-i_1> \ell(n)$ and $0\leq j_2-j_1\le \ell(n)$, i.e.\ \emph{a half-separated case}, we have
  \begin{align*}
    & \sum_{\substack{i_1, i_2\in [0, {n})\\ i_2-i_1> \ell(n)}} \
    \sum_{\substack{j_1, j_2\in [0, {n})\\  0\le j_2-j_1\le
    \ell(n)}} \int \int \1_{B(T_2^{j_1}y, r_{n})}
    (T_1^{i_1}x)\cdot \1_{B(T_2^{j_2}y, r_{n})}
    (T_1^{i_2}x) \, d\mu_1(x)d\mu_2(y)\\
    & \quad = \sum_{\substack{i_1, i_2, j_1, j_2 \in [0, n)\\
    i_2-i_1> \ell(n)\\ 0\le j_2-j_1\le \ell(n)}} \int\int\1_{B(T_2^{j_1}y,
    r_{n})}(x)\cdot \1_{B(T_2^{j_2}y,
    r_{n})}(T_1^{i_2-i_1}x) \, d\mu_1(x)d\mu_2(y)\\ 
    & \quad \le \sum_{\substack{i_1, i_2\in [0, n)\\ i_2-i_1>
    \ell(n)}} \ \sum_{\substack{j_1, j_2\in [0, n)\\  0\le
    j_2-j_1\le \ell(n)}}  \biggl( \int\mu_1 (B(T_2^{j_1}y,
    r_{n})) \mu_1 (B(T_2^{j_2}y,
    r_{n})) \, d\mu_2(y) + \\
    & \hspace{10.5cm} + \mathcal{O}\left( \e^{-\theta\ell(n)}\right) \biggr)\\
    &\quad \le   \sum_{\substack{i_1, i_2\in [0, n)\\ i_2-i_1>
    \ell(n)}}  \ \sum_{\substack{j_1, j_2\in [0, n)\\  0\le
    j_2-j_1\le \ell(n)}} \biggl(\int \mu_1 (B (y, r_{n}))^2 \,
    d\mu_2(y) + \mathcal{O}\left( \e^{-\theta\ell(n)}\right)\biggr)\\ 
&\quad \le n^{3}\ell(n) \biggl(\int \mu_1 (B(y,
r_{n}))^2 \, d\mu_2(y) + \mathcal{O}\left( \e^{-\theta\ell(n)}\right) \biggr),  
  \end{align*} where in the penultimate line we have used the Cauchy-Schwarz inequality and
  $T_2$-invariance of $\mu_2$. In the case where $0\leq i_2-i_1\leq  \ell(n)$ and $j_2-j_1> \ell(n),$ by an analogous argument we obtain:
 \begin{align*} 
 	&\sum_{\substack{i_1, i_2\in [0, {n})\\ 0\leq i_2-i_1\leq \ell(n)}} \
  \sum_{\substack{j_1, j_2\in [0, {n})\\  j_2-j_1>
  		\ell(n)}} \int \int \1_{B(T_2^{j_1}y, r_{n})}
  (T_1^{i_1}x)\cdot \1_{B(T_2^{j_2}y, r_{n})}\\
  &\hspace{3cm}\le n^{3}\ell(n) \biggl(\int \mu_2 (B(y,
  r_{n}))^2 \, d\mu_1(y) + \mathcal{O}\left( \e^{-\theta\ell(n)}\right) \biggr).
  \end{align*} Similar estimates can be obtained in the half separated case when $i_{2}\leq i_{1}$ and $j_{2}\leq j_{1}$. In summary, in the half-separated case we have the following bound:
  \begin{align}
  	\label{eq:Third part of variance bound}
  	&\sum_{\substack{i_1, i_2\in [0, {n})\\ |i_2-i_1|> \ell(n)}} \
  	\sum_{\substack{j_1, j_2\in [0, {n})\\  0\le |j_2-j_1|\le
  			\ell(n)}} \int \int \1_{B(T_2^{j_1}y, r_{n})}
  	(T_1^{i_1}x)\cdot \1_{B(T_2^{j_2}y, r_{n})}
  	(T_1^{i_2}x) \, d\mu_1(x)d\mu_2(y)\nonumber\\
  	&+\sum_{\substack{i_1, i_2\in [0, {n})\\ 0\leq |i_2-i_1|\leq \ell(n)}} \
  	\sum_{\substack{j_1, j_2\in [0, {n})\\  |j_2-j_1|>
  			\ell(n)}} \int \int \1_{B(T_2^{j_1}y, r_{n})}
  	(T_1^{i_1}x)\cdot \1_{B(T_2^{j_2}y, r_{n})}
  	(T_1^{i_2}x) \, d\mu_1(x)d\mu_2(y)\nonumber\\
  	&\lesssim  n^{3}\ell(n) \biggl(\int \mu_1 (B(y,
  	r_{n}))^2 \, d\mu_2(y) +\mathcal{O}\left( \e^{-\theta\ell(n)}\right) \biggr)\\
  	&\quad +n^{3}\ell(n) \biggl(\int \mu_2 (B(y,
  	r_{n}))^2 \, d\mu_1(y) + \mathcal{O}\left( \e^{-\theta\ell(n)}\right)\biggr).\nonumber
  \end{align}

Combining \eqref{eq:First part of variance bound}, \eqref{eq:Second part of variance bound}, and \eqref{eq:Third part of variance bound} we have  
  \begin{align*}
\frac{\E(S_n^2)-\E( S_n)^2}{\E( S_n)^2}&\lesssim \frac{n^{4} \e^{-\theta\ell(n)}}{\E( S_n)^2}+\frac{\ell(n)^2 \E( S_n)}{\E( S_n)^2}+ \frac{n^{3}\ell(n)\int \mu_1 (B(y, r_{n}))^2 \, d\mu_2(y)}{\E( S_n)^2} \\
 &\quad+ \frac{n^{3}\ell(n) \int \mu_2 (B(y,
 	r_{n}))^2 \, d\mu_1(y) }{\E( S_n)^2}+\frac{n^{3} \e^{-\theta\ell(n)}\ell(n)}{\E( S_n)^2}. 
	  \end{align*}
We now consider each of the terms on the right hand side of the above and show that they are each bounded above by a constant multiple of $n^{-\hat\gamma/2}$. This will in turn imply that \eqref{eq:mainest} holds.   The first and fifth terms decay to zero faster than any polynomial by our choice of function $\ell$. Thus these two terms can be bounded above by a constant multiple of $n^{-\hat\gamma/2}.$

For the second term, using \eqref{eq:Expectation lower bound two orbit} we have
  \begin{equation*}
  	\frac{\ell(n)^2 \E(S_n)}{\E(S_n)^2}  \lesssim \frac{\ell(n)^2 }{ n^{\gamma}}\lesssim \frac{1}{ n^{\frac{\hat\gamma}{2}}}.
  \end{equation*}
So this term satisfies the desired upper bound.
  
 For the third term we apply Lemma~\ref{lem:tight} for $(i,j) = (1,2)$ and use \eqref{eq:r_ncond} and \eqref{eq:slowgrow} to obtain
  \begin{align*}
  	&\frac{n^{3}\ell(n)\int \mu_1 (B(y, r_{n}))^2 \, d\mu_2(y)}{\E(S_n)^2} \\
  	&\quad\lesssim \frac{n^{3}\ell(n) \bigl( \int\mu_1(B(y, r_{n})) \,
  		d\mu_1(y) \bigr)^{\frac12} \int\mu_1(B(y, r_{n})) \,
  		d\mu_2(y)}{n^{4} \biggl(\int\mu_1(B(y, r_{n})) \,
  		d\mu_2(y)\biggr)^2} \\
  	&\quad= \frac{\ell(n) \left( n^2\int\mu_1(B(y, r_{n})) \, d\mu_1(y)
  		\right)^{\frac12}}{n^2 \int\mu_1(B(y, r_{n})) \,
  		d\mu_2(y)} \lesssim \frac{n^{\frac\alpha2}\ell(n)}{n^\gamma} \lesssim \frac{1}{n^{\frac{\hat\gamma}2}} .\\
  \end{align*}  
 So the third term satisfies the desired decay rate. By an analogous argument it can be shown that the fourth term is also bounded above by a constant multiple of $n^{-\hat\gamma/2}.$ This completes our proof.   
\end{proof}

\section{One orbit case: conditioning and expectation}
\label{sec:one orbit first part}

In this section and the next we fix $(X,T,\mu)$ satisfying the assumptions of Theorem \ref{mainthm:1orb}. We also fix $\beta>0$ satisfying the inequalities listed in this theorem. By the definition of Frostman dimension and the early return exponent, we can choose $\tilde{F}_{\mu},\tilde{D}_{\mu},C>0$ so that the following properties hold:
\begin{itemize}
	\item $\mu(B(x,r))\leq Cr^{\tilde{F}_{\mu}}$ for all $x\in X$ and $r>0.$
	\item  $\mu(A_{r}(n))\leq Cr^{\tilde{D}_{\mu}}$ for all $r>0$ and $n\in \N$. 
	\item $\beta(C_{\mu}-\tilde{D}_{\mu})<1.$
	\item $\beta(C_{\mu}-\tilde{F}_{\mu})<1$.
\end{itemize} We now set out to prove that $\mu$-almost every $x$ satisfies
$$\lim_{n\to \infty}  \frac{ \#\left\{0\le i\neq  j<n:|T^ix- T^jx|\le \frac s{n^{\beta}}\right\}}{n^2\int\mu\left(B\left(z,\frac s{n^{\beta}}\right)\right)~d\mu(z)}= 1$$ for all $s>0$. Assuming we have fixed $s>0$ we associate the random variable $$S_{n}(x)=\sum_{0\leq i\neq j<n}\1_{B(T^ix, \frac{s}{n^{\beta}})}(T^jx).$$ We will prove Theorem \ref{mainthm:1orb} via a second moment analysis of $S_{n}$ as in the proof of Theorem~\ref{thm:2orbgen}. However, in this case, because we are now considering a single orbit instead of two orbits, we have lost a degree of independence that was crucial in our earlier calculations, so for example, even computing the expectation of $S_n$ becomes a significant task. To overcome this issue it is necessary to introduce a conditioning argument. We will outline this argument in the next subsection.  In this section we use it to estimate the expectation, or first moment, of our $S_n$.  In the next section we will estimate the second moment and give the proof of Theorem \ref{mainthm:1orb}.

\subsection{A conditioning argument} 
\label{ssec:conditioning} Given $s>0$, in what follows we define the sequence $(r_n)_n$ via the equation $$r_{n}=\frac{s}{n^{\beta}}$$ for all $n\in \N$. We suppress the dependence on $s$ in this notation. When we refer to a sequence $(r_n)$ we are implicitly making a choice of $s$.

The conditioning argument we will use throughout our proof involves introducing a parameter $\epsilon>0,$ and for each $n\in \N$ we let $m\in \N$ be the unique integer satisfying
\begin{equation}
	\label{eq:define m}
	2^{-m}<r_{n}^{1+\epsilon}\leq 2^{-m+1}.
\end{equation} In our notation we will suppress the dependence of $m$ on $\epsilon$. Our conditioning is provided by the dyadic partition provided by intervals of size $2^{-m}$. With this in mind we introduce the following notation: let 
$$I_{0}:=\left[0,\frac{1}{2^{m}}\right] \textrm{ and }I_{k}:=\left(\frac{k}{2^{m}},\frac{k+1}{2^{m}}\right]\textrm{ for }1\leq k\leq 2^{m}-1.$$ Given $\epsilon>0$ and $0\leq k\leq 2^{m}-1$ we also let $$B_{k}^{+}:=B\left(\frac{k}{2^{m}},r_{n}+\frac{1}{2^{m}}\right) \textrm{ and }B_{k}^{-}:=B\left(\frac{k}{2^{m}},r_{n}-\frac{1}{2^{m}}\right).$$ Again we suppress $\epsilon$ from our notation.

We observe that the partition $\{I_{k}\}_{k=0}^{2^m-1}$ has the property that for all $x, y\in I$ we have
\begin{equation}
	\label{eq:binary conditioning}
\1_{I_{k}}(x)\cdot \1_{B_{k}^{-}}(y) \le 
\1_{I_{k}}(x)\cdot \1_{B(x, r_n)}(y)  \le \1_{I_{k}}(x)\cdot \1_{B_{k}^{+}}(y).
\end{equation}
Moreover, it is an obvious fact that for $x\in [0,1]$ and $i, j\in \N_0$ we have
\begin{equation}
	\label{eq:recurrence conditioning}
\1_{B(T^ix, r_n)}(T^jx) = \sum_{k=0}^{2^m-1} \1_{I_{k}}(T^{i}x)\cdot \1_{B(T^{i}x, r_n)}(T^jx).
\end{equation} Combining \eqref{eq:binary conditioning} and \eqref{eq:recurrence conditioning} yields
\begin{equation}
	\label{eq:final conditioning}
\sum_{k=0}^{2^{m-1}}\1_{I_{k}}(T^{i}x)\1_{B_{k}^{-}}(T^{j}x)\leq \1_{B(T^ix, r_n)}(T^jx)\leq \sum_{k=0}^{2^{m-1}}\1_{I_{k}}(T^{i}x)\1_{B_{k}^{+}}(T^{j}x).
\end{equation}
Crucially the left hand side and right hand side of \eqref{eq:final conditioning} are in a form where our mixing assumptions can be applied. This is the important property that our conditioning approach provides.

We also highlight the following straightforward counting bounds that will be used throughout our proofs. Suppose that $\epsilon>0$ and that $m$ is defined by \eqref{eq:define m}, then for any $x\in [0,1]$ we have
\begin{equation}
	\label{eq:counting bounds}
	\sum_{k=0}^{2^{m}-1}\1_{B_{k}^{+}}(x)\lesssim 2^{m}r_{n}\lesssim n^{\beta\epsilon}.
\end{equation} We will also regularly use the trivial fact that for any $x\in [0,1]$ we have
\begin{equation}
	\label{eq:trivial counting bounds}
	\sum_{k=0}^{2^{m}-1}\1_{I_{k}}(x)=1.
\end{equation}
We will also need the following technical lemma which controls the error when $\int\mu(B(x,r_n))~d\mu(x)$ is approximated by a sum over the dyadic partition corresponding to a choice of $\epsilon$ and $m$. In the proof of this lemma and at many points later on in our discussion, we will need the following equality which follows from the definition of correlation dimension:
\begin{equation}
	\label{eq:integral asymptotics}
	\int \mu(B(x,r_{n}))~d\mu(x)=n^{-\beta C_{\mu}+o(1)}.
\end{equation}

\begin{lemma}
	\label{lem:Riem1} 
Let $s>0$.	Given $\epsilon>0$ and $m$ defined by \eqref{eq:define m}, we have the following
\begin{align*}
	\left| \sum_{k=0}^{2^m-1} \mu(I_{k}) \mu\left(B_{k}^{+}\right)- \int\mu(B(x,r_n))~d\mu(x) \right|&\lesssim\int \mu\left(B\left(x,\frac{8}{2^{m}}\right)\right)\, d\mu(x)\\
	& \lesssim n^{-\beta C_{\mu}(1+\epsilon/2)}
	\end{align*}
and 
\begin{align*}
	\left| \sum_{k=0}^{2^m-1} \mu(I_{k}) \mu\left(B_{k}^{-}\right)- \int\mu(B(x,r_n))~d\mu(x) \right|&\lesssim \int \mu\left(B\left(x,\frac{8}{2^{m}}\right)\right)\, d\mu(x)\\
	& \lesssim n^{-\beta C_{\mu}(1+\epsilon/2)}.
\end{align*}
Moreover, we always have the bound $$\sum_{k=0}^{2^m-1} \mu(I_{k}) \mu\left(B_{k}^{+}\right)\lesssim n^{-\beta C_{\mu}+o(1)}.$$
\end{lemma}

\begin{proof}
	In the proof below we just prove the first statement for the set of balls $\{B_{k}^{+}\}_{k=0}^{2^{m}-1}$. The second statement follows by an analogous argument. Our final bound is a consequence of our first statement and \eqref{eq:integral asymptotics}.

Let $g_{m}:[0,1]\to \{\frac{k}{2^{m}}\}_{k=0}^{2^m-1}$ be the function given by the rule $g_{m}(x)=\frac{k}{2^{m}}$ if $x\in I_{k}.$ Then 
\begin{align*}
	&\left|\sum_{k=0}^{2^m-1} \mu\left(I_{k}\right) \mu\left(B_{k}^{+}\right)- \int\mu(B(x,r_n))~d\mu(x)\right|\\
&= \left|\int \mu\left(B\left(g_{m}(x),r_{n}+ \frac{1}{2^{m}}\right)\right)-\mu(B(x,r_{n}))\, d\mu(x)\right|\\
&=\int \mu\left(B\left(g_{m}(x),r_{n}+ \frac{1}{2^{m}}\right)\setminus B(x,r_{n})\right)\, d\mu(x)\\
&\leq \int \mu\left(B\left(g_{m}(x)+r_{n},\frac{1}{2^{m}}\right)\right)+\mu\left(B\left(g_{m}(x)-r_{n},\frac{1}{2^{m}}\right)\right)\, d\mu. 
\end{align*} Thus to prove our result it remains to establish suitable upper bounds for the two integrals appearing in the final line of the above. We will just show how to bound the first of these integrals. The other is handled similarly. Let $p\in \N$ be the unique integer satisfying $$\frac{p}{2^{m}}\leq r_{n}< \frac{p+1}{2^{m}}.$$ Then we have the following 
\begin{align}
	\label{eq:measure of annulus ends}
&\int \mu\left(B\left(g_{m}(x)+r_{n},\frac{1}{2^{m}}\right)\right)\, d\mu(x)\nonumber\\
&\leq \sum_{k=0}^{2^{m}-1}\mu\left(I_{k}\right)\mu\left(B\left(\frac{k+p}{2^{m}},\frac{3}{2^{m}}\right)\right)\nonumber\\
&\leq \left(\sum_{k=0}^{2^{m}-1}\mu\left(I_{k}\right)^{2}\right)^{1/2}\left(\sum_{k=0}^{2^{m}-1}\mu\left(B\left(\frac{k+p}{2^{m}},\frac{3}{2^{m}}\right)\right)^{2}\right)^{1/2}\quad (\text{Cauchy-Schwarz})\nonumber\\
&\leq \left(\sum_{k=0}^{2^{m}-1}\mu\left(I_{k}\right)^{2}\right)^{1/2}\left(\sum_{k=0}^{2^{m}-1}\mu\left(B\left(\frac{k}{2^{m}},\frac{3}{2^{m}}\right)\right)^{2}\right)^{1/2}.
\end{align}
We now bound the two terms appearing in \eqref{eq:measure of annulus ends}. The following bound is obvious
\begin{equation}
	\label{eq:annulusterma}
\left(\sum_{k=0}^{2^{m}-1}\mu\left(I_{k}\right)^{2}\right)^{1/2}\leq \left(\int \mu\left(B\left(x,\frac{8}{2^{m}}\right)\right)\, d\mu(x)\right)^{1/2}.
\end{equation}
For the second term we have 
\begin{align}
	\label{eq:annulustermb}
\sum_{k=0}^{2^{m}-1}\mu\left(B\left(\frac{k}{2^{m}},\frac{3}{2^{m}}\right)\right)^{2}&=\sum_{k=0}^{2^{m-1}}\sum_{q=-3}^{2}\mu(I_{k+q})\mu\left(B\left(\frac{k}{2^{m}},\frac{3}{2^{m}}\right)\right)\nonumber\\
&\leq \sum_{k=0}^{2^{m-1}}\sum_{q=-3}^{2}\mu(I_{k+q})\mu\left(B\left(\frac{k+q}{2^{m}},\frac{7}{2^{m}}\right)\right)\nonumber\\
&\leq 6 \sum_{k=0}^{2^{m}-1}\mu(I_{k})\mu\left(B\left(\frac{k}{2^{m}},\frac{7}{2^{m}}\right)\right)\nonumber\\
&\leq 6\int \mu\left(B\left(x,\frac{8}{2^{m}}\right)\right)~d\mu(x).
\end{align}
Substituting \eqref{eq:annulusterma} and \eqref{eq:annulustermb} into \eqref{eq:measure of annulus ends}, and using the definition of correlation dimension and \eqref{eq:define m} now yields
\begin{align*}
\int \mu\left(B\left(g_{m}(x)+r_{n},\frac{1}{2^{m}}\right)\right)\, d\mu(x)&\leq 7\int \mu\left(B\left(x,\frac{8}{2^{m}}\right)\right)~d\mu(x)\\
&\lesssim  n^{-\beta C_{\mu}(1+\epsilon/2)}.
\end{align*}This completes our proof.
\end{proof}

\subsection{Expectation estimate}

We next give the main result from this section, an estimate on $\E(S_n)$.  Recall that
$\beta, \tilde{D}_{\mu}, \tilde{F}_{\mu}$ and $C_{\mu}$ satisfy 
$$\beta(C_{\mu}-\tilde{D}_{\mu})<1,\,
\beta(C_{\mu}-\tilde{F}_{\mu})<1,\,
\beta C_{\mu}<2.$$ 

\begin{proposition}
\label{prop:expectation}
There exists $\epsilon_{1}>0$ such that for any $s>0$ we have
\begin{align*}
	\left|\E(S_n)-n^{2}\int \mu(B(x,r_{n}))\,d\mu(x)\right|&\lesssim n^{2-\beta C_{\mu}-\epsilon_{1}}.
	\end{align*}
\end{proposition}

\begin{proof}We start by stating the following simple bound and equality:
	\begin{equation}
		\label{eq:expectation lower bound}
	\sum_{\substack{i, j \in [0, {n})\\ |j-i| \geq\ell(n)}}\int \1_{B(T^i(x), r_n)} (T^jx)\, d\mu(x)\leq \E(S_{n})
	\end{equation}and 
	\begin{align}
		\label{eq:expectation equality}
	\E(S_{n})&= \sum_{\substack{i, j \in [0, {n})\\ |j-i| \geq\ell(n)}}\int \1_{B(T^i(x), r_n)} (T^jx)\, d\mu(x)\\
	&\quad \quad+\sum_{\substack{i, j \in [0, {n})\\ |j-i| <\ell(n)}}\int \1_{B(T^i(x), r_n)} (T^jx)\, d\mu(x)\nonumber.
	\end{align}
	Thus to prove our result it is sufficient to obtain suitably good upper and lower bounds for the sum over the well separated indices, and an upper bound for the sum over close indices. We begin by proving an upper bound for the summation over well separated indices. Let $\epsilon>0$ and $m$ be defined by \eqref{eq:define m}.
 By our exponential $4$-mixing on intervals assumption, the right hand inequality in \eqref{eq:final conditioning}, and Lemma \ref{lem:Riem1} we have:
\begin{align}
	\label{eq:Separated bound1}
	&\sum_{\substack{i, j \in [0, {n})\\ |j-i| \geq\ell(n)}} \int  \1_{B(T^i(x), r_n)} (T^jx) ~d\mu(x)\nonumber\\
	&\quad\leq   \sum_{\substack{i, j \in [0, {n})\\ |j-i| \geq\ell(n)}}\sum_{k=0}^{2^m-1} \int  \1_{I_{k}}(T^ix)\cdot \1_{B_{k}^{+}}(T^jx) ~d\mu(x)\nonumber
	\\
	&\quad=   \sum_{\substack{i, j \in [0, {n})\\ |j-i| \geq\ell(n)}}\sum_{k=0}^{2^m-1}\left(  \mu\left(I_{k}\right) \mu\left(B_{k}^{+}\right) + \mathcal{O}(\e^{-\theta\ell(n)})\right)\nonumber\\ 
	 & \quad\leq \sum_{\substack{i, j \in [0, {n})\\ |j-i| \geq\ell(n) }}\left(\int\mu(B(x,r_{n}))\, d\mu(x)+\mathcal{O}\left(n^{-\beta C_{\mu}(1+\epsilon/2)}\right)\right)+\mathcal{O}(2^{m}\e^{-\theta\ell(n)})\nonumber\\
	 &\quad\leq \sum_{\substack{i, j \in [0, {n})\\ |j-i| \geq\ell(n) }}\int\mu(B(x,r_{n}))\, d\mu(x)+\mathcal{O}\left(n^{2-\beta C_{\mu}(1+\epsilon/2)}\right).
\end{align}  In the final line we have used that $2^{m}\e^{-\theta\ell(n)}$ decays to zero faster than any polynomial (recall Definition~\ref{def:ell}), and so can be absorbed into our final error term. It can similarly be shown using the first inequality in \eqref{eq:final conditioning} that 
\begin{align}
	\label{eq:separated bound2}&\sum_{\substack{i, j \in [0, {n})\\ |j-i| \geq\ell(n)}} \int  \1_{B(T^i(x), r_n)} (T^jx) ~d\mu(x)\\
	 &\hspace{2cm} \geq\sum_{\substack{i, j \in [0, {n})\\ |j-i| \geq\ell(n) }}\int\mu(B(x,r_{n}))\, d\mu(x)+\mathcal{O}\left(n^{2-\beta C_{\mu}(1+\epsilon/2)}\right)\nonumber.
	\end{align} Combining \eqref{eq:integral asymptotics}, \eqref{eq:Separated bound1}, \eqref{eq:separated bound2},  and using that $$\#\{i,j\in [0,n):|j-i|\geq \ell(n)\}=n^{2}+\mathcal{O}(n\ell(n)),$$  we have that there exists $\epsilon_{1}'>0$ sufficiently small that does not depend on our choice of $s$ such that
	\begin{align}
		\label{eq:finalseparated bound}
		&\sum_{\substack{i, j \in [0, {n})\\ |j-i| \geq\ell(n)}} \int  \1_{B(T^i(x), r_n)} (T^jx) ~d\mu(x) =n^{2}\int\mu(B(x,r_{n}))\, d\mu(x) +\mathcal{O}(n^{2-\beta C_{\mu}-\epsilon_{1}'}).
	\end{align}

	We now need to bound the contribution coming from those indices satisfying $|j-i|<\ell(n).$ Using the $T$-invariance of $\mu$ and the inequality $\beta(C_{\mu}-\tilde{D}_{\mu})<1$, it follows that there exists $\epsilon_{2}'>0$ sufficiently small such that does not depend on $s$ such that the following holds:
\begin{align}
	\label{eq:Close bound}
	&\sum_{\substack{i, j \in [0, {n})\\ |j-i| <\ell(n)}}\int\1_{B(T^i(x), r_n)} (T^jx)\, d\mu(x) \nonumber\\
	&\quad\leq \sum_{\substack{i, j \in [0, {n}),\, i>j\\ |j-i| <\ell(n)}}\int\1_{A_{r_{n}}(i-j)} (T^jx)\, d\mu(x)\nonumber\\
	&\hspace{3cm} +\sum_{\substack{i, j \in [0, {n}),\, j>i\\ |j-i| <\ell(n)}}\int\1_{A_{r_{n}}(j-i)} (T^ix)\, d\mu(x)\nonumber\\
	&\quad\leq \sum_{\substack{i, j \in [0, {n}),\, i>j\\ |j-i| <\ell(n)}}\int\1_{A_{r_{n}}(i-j)} (x)\, d\mu(x)\nonumber\\
	& \hspace{3cm}+\sum_{\substack{i, j \in [0, {n}),\, j>i\\ |j-i| <\ell(n)}}\int\1_{A_{r_{n}}(j-i)} (x)\, d\mu(x)\nonumber\\
	&\quad =\sum_{\substack{i, j \in [0, {n})\\ |j-i| <\ell(n)}}\mu(A_{r_{n}}(|j-i|)) \lesssim r_{n}^{\tilde{D}_{\mu}}\cdot  n\ell(n) \lesssim n^{1-\tilde{D}_{\mu}\beta}\cdot  \ell(n)\nonumber\\
	&\quad \lesssim  n^{2-\beta C_{\mu}-\epsilon_{2}'}. 
\end{align}
In the penultimate line we have used the definition of $r_{n}$. In the final line we have used that $\beta(C_{\mu}-\tilde{D}_{\mu})<1$ is equivalent to $2-\beta C_{\mu}>1-\tilde{D}_{\mu}\beta$, and that $\ell(n) = n^{o(1)}$.  Combining \eqref{eq:finalseparated bound} and \eqref{eq:Close bound} it is clear that our result holds if we take $\epsilon_{1}=\min \{\epsilon_{1}',\epsilon_{2}'\}$.
\end{proof}

\section{The one orbit case: second moment argument}
\label{sec:one orbit second moment argument}

We now turn our attention to $\E(S_{n}^{2})$. We start by observing that
$$\E(S_{n}^{2})=4\sum_{i_1< j_1\in [0, n)} \sum_{i_2< j_2\in [0, n)} \int \1_{B(T^{i_1}(x), r_n)} (T^{j_1}x) \1_{B(T^{i_2}(x), r_n)} (T^{j_2}x)~d\mu(x).$$  When analysing this expression for $\E(S_{n}^{2})$ it is useful to consider the ordering of the quadruple $(i_{1},i_{2},j_{1},j_{2})$. We start by highlighting that because our indices always satisfy $i_1<j_1$ and $i_2<j_2,$ there are six possible orderings for these indices:
 \begin{align*}
 	&i_1\leq j_1\leq i_2\leq j_2\tag{Case A}\\
 	&i_1\leq i_{2}\leq j_1\leq j_2\tag{Case B}\\
 	&i_1\leq i_{2}\leq j_{2}\leq j_1\tag{Case C}\\
 	&i_2\leq j_2\leq i_1\leq j_1\tag{Case D}\\
 	&i_2\leq i_{1}\leq j_2\leq j_1\tag{Case E}\\
 	&i_2\leq i_{1}\leq j_{1}\leq j_2\tag{Case F}.
 \end{align*} We emphasise that after relabelling we see that Case A is equivalent to Case D, Case B is equivalent to Case E, and Case C is equivalent to Case F. In what follows whenever we speak of a quadruple of indices $(i_{1},i_{2},j_{1},j_{2})$ we will always implicitly assume that $i_{1}<j_{1}$ and $i_{2}<j_{2}$. In addition to focusing on these cases it will be necessary to consider the gaps between successive indices. 
 
 \subsection{Well separated indices}
 We say that a quadruple $(i_{1},i_{2},j_{1},j_{2})\in [0,n)^{4}$ is \emph{well separated} if all terms are separated by a factor at least $\ell(n),$ i.e. 
 $$\min\left\{|i_{1}-i_{2}|, |i_{1}-j_{1}|, |i_{1}-j_{2}|,|i_{2}-j_{1}|,|i_{2}-j_{2}|,|j_{1}-j_{2}|\right\}\geq \ell(n).$$
  The following proposition considers the case when the indices are well separated, and as we will see, it captures the dominant term in our expansion of $\E(S_{n}^{2})$.
 
 \begin{proposition}
 	\label{prop:Well separated} 
 	There exists $\epsilon_{2}>0$ such that for any $s>0,$ if $(i_{1},i_{2},j_{1},j_{2}) $ are well separated then
 \begin{align*}
 	&\left|\int \1_{B(T^{i_1}(x), r_n)} (T^{j_1}x) \1_{B(T^{i_2}(x), r_n)} (T^{j_2}x)~d\mu(x)-\left(\int \mu(B(x,r_{n}))\,d\mu(x)\right)^{2}\right|\\
 	& \hspace{8cm}\lesssim  n^{-2\beta C_{\mu}-\epsilon_{2} }.
 \end{align*}
 \end{proposition}
 \begin{proof}
 In this proof we will assume that we are in Case A. The other cases are handled similarly.	We first establish a suitable upper bound. Let $\epsilon>0$ and $m$ be defined by \eqref{eq:define m}. The following inequality is a consequence of the right hand side of \eqref{eq:final conditioning}:
 \begin{align}
 \label{eq:Well separated conditioning}&\int \1_{B(T^{i_1}(x), r_n)} (T^{j_1}x) \1_{B(T^{i_2}(x), r_n)} (T^{j_2}x)~d\mu(x)\nonumber\\
 	&\leq \sum_{k=0}^{2^m-1}\sum_{l=0}^{2^{m}-1} \int  \1_{I_{k}}(T^{i_{1}}x)\cdot \1_{I_{l}}(T^{i_{2}}x)\cdot \1_{B_{k}^{+}}(T^{j_{1}}x)\cdot \1_{B_{l}^{+}}(T^{j_{2}}x) ~d\mu(x).
 	\end{align}
 	Let us now fix $0\leq k\leq 2^{m-1}$ and $0\leq l\leq 2^{m-1}$. 
 	Repeatedly using the $T$-invariance of $\mu$ and our exponential mixing for BV against $L^{\infty}$ assumption, we have
 	\begin{align*}
 	&\int \1_{I_{k}}(T^{i_{1}}x)\cdot \1_{I_{l}}(T^{i_{2}}x)\cdot \1_{B_{k}^{+}}(T^{j_{1}}x)\cdot \1_{B_{l}^{+}}(T^{j_{2}}x) ~d\mu(x)\\
 	&\quad=\int \1_{I_{k}}(x)\cdot \1_{I_{l}}(T^{i_{2}-i_1}x)\cdot \1_{B_{k}^{+}}(T^{j_{1}-i_1}x)\cdot \1_{B_{l}^{+}}(T^{j_{2}-i_1}x) ~d\mu(x)\\
 		&\quad= \mu\left(I_{k}\right) \int  \1_{I_{l}}(x)\cdot \1_{B_{k}^{+}}(T^{j_{1}-i_2}x)\cdot \1_{B_{l}^{+}}(T^{j_{2}-i_2}x) ~d\mu(x)+\mathcal{O}\left(\e^{-\theta\ell(n)}\right)\\
 		&\quad= \mu\left(I_{k}\right)\mu\left(I_{l}\right) \int   \1_{B_{k}^{+}}(x)\cdot \1_{B_{l}^{+}}(T^{j_{2}-j_1}x)~d\mu(x)+\mathcal{O}\left(\e^{-\theta\ell(n)}\right)\\
 		&\quad=  \mu(I_{k})\mu(I_{l})\mu(B_{k}^{+})\mu(B_{l}^{+})+\mathcal{O}\left(\e^{-\theta\ell(n)}\right).
 		\end{align*}
 		Using the equality above, together with Lemma \ref{lem:Riem1}, \eqref{eq:integral asymptotics} and \eqref{eq:Well separated conditioning} yields 
 		\begin{align*}
 		&\int \1_{B(T^{i_1}(x), r_n)} (T^{j_1}x) \1_{B(T^{i_2}(x), r_n)} (T^{j_2}x)~d\mu(x)\\
 		&\quad\leq \sum_{k=0}^{2^{m-1}}\sum_{l=0}^{2^{m}-1}\left(\mu(I_{k})\mu(I_{l})\mu(B_{k}^{+})\mu(B_{l}^{+})+\mathcal{O}\left(\e^{-\theta\ell(n)}\right)\right)\\
 		&\quad=\left(\int \mu(B(x,r_{n}))\,d\mu(x)+\mathcal{O}(n^{-\beta C_{\mu}(1+\epsilon/2)})\right)^{2}+\mathcal{O}\left(4^{m}\e^{-\theta\ell(n)}\right)\\
 		&\quad=\left(\int \mu(B(x,r_{n}))\,d\mu(x)\right)^{2}+\mathcal{O}(n^{-\beta C_{\mu}(2+\epsilon/4)})+\mathcal{O}\left(4^{m}\e^{-\theta\ell(n)}\right).
 		\end{align*}
 		Using that our second error term decays to zero faster than any polynomial, it follows from the above that there exists $\epsilon_{1}'>0$ that does not depend on our choice of $s$ such that 
 		\begin{align}
 			\label{eq:well separated upper bound}
 			&\int \1_{B(T^{i_1}(x), r_n)} (T^{j_1}x) \1_{B(T^{i_2}(x), r_n)} (T^{j_2}x)~d\mu(x)\nonumber\\
 			&\quad\leq \left(\int \mu(B(x,r_{n}))\,d\mu(x)\right)^{2}+\mathcal{O}(n^{-2\beta C_{\mu}-\epsilon_{1}'}).
 		\end{align} It can similarly be shown using the left hand side of \eqref{eq:final conditioning} that there exists $\epsilon_{2}'>0$ that does not depend on $s$ such that 
 		\begin{align}
 			\label{eq:well separated lower bound}
 			&\int \1_{B(T^{i_1}(x), r_n)} (T^{j_1}x) \1_{B(T^{i_2}(x), r_n)} (T^{j_2}x)~d\mu(x)\nonumber\\
 			&\quad\geq \left(\int \mu(B(x,r_{n}))\,d\mu(x)\right)^{2}+\mathcal{O}(n^{-2\beta C_{\mu}-\epsilon_{2}'}).
 		\end{align} Together \eqref{eq:well separated upper bound} and \eqref{eq:well separated lower bound} imply our result. 
 \end{proof}

 \subsection{Not well separated indices}
We now consider what happens when our indices are not well separated. For our purposes, we just need to bound from above the contribution to $\E(S_{n}^{2})$ coming from those terms. We proceed via a case analysis based on how many large gaps there are between successive indices. Before doing this it is necessary to introduce some terminology. Given a quadruple of indices $(i_{1},i_{2},j_{1},j_{2})$ we inductively define
 $$\gamma_{1}=\min \{i_{1},i_{2},j_{1},j_{2}\}$$ and $$\gamma_{i}=\min\{\{i_{1},i_{2},j_{1},j_{2}\}\setminus \cup_{l=1}^{i-1}\{\gamma_{l}\}\}\}$$ for $2\leq i\leq 4$. When there isn't a unique minimum then we choose arbitrarily. Notice that the set $\{\gamma_{i}\}$ satisfies $\{\gamma_{1},\gamma_{2},\gamma_{3},\gamma_{4}\}=\{i_{1},i_{2},j_{1},j_{2}\}$ and $$\gamma_{1}\leq \gamma_{2}\leq \gamma_{3}\leq \gamma_{4}.$$ It is this second property that will be more useful for our purposes. The cases we will study are defined as follows.
 \begin{itemize}
 	\item We say that $(i_{1},i_{2},j_{1},j_{2})$ has \emph{no gaps} if $$\gamma_{i+1}-\gamma_{i}<\ell(n)$$ for all $1\leq i\leq 3.$
 	\item We say that $(i_{1},i_{2},j_{1},j_{2})$ has \emph{one gap} if there exists a unique $1\leq i\leq 3$ such that $$\gamma_{i+1}-\gamma_{i}\geq \ell(n).$$
 	\item We say that $(i_{1},i_{2},j_{1},j_{2})$ has \emph{two gaps} if there exists precisely two indices $i,j\in \{1,2,3\}$ satisfying
 	$$\gamma_{i+1}-\gamma_{i}\geq \ell(n)\textrm{ and }\gamma_{j+1}-\gamma_{j}\geq \ell(n). $$
 \end{itemize} We emphasise that a quadruple $(i_{1},i_{2},j_{1},j_{2})$ is either well separated, has no gaps, has one gap, or has two gaps. We will obtain bounds for each of these cases in turn.  We recall again, for use in the next propositions, that
$\beta, \tilde{D}_{\mu}$ and $\tilde{F}_{\mu}$ satisfy $$\beta(C_{\mu}-\tilde{D}_{\mu})<1,\,
\beta(C_{\mu}-\tilde{F}_{\mu})<1,\,
\beta C_{\mu}<2.$$ 
 
 \subsubsection{The no gaps case}
  In the case of no gaps we have the following bound:
 
 \begin{proposition}
 	\label{prop:no gap prop}
There exists $\epsilon_{3}>0$ such that for any $s>0$ we have
\begin{align*}
	&\sum_{\substack{(i_1, i_2, j_1, j_2)\\ (i_1, i_2, j_1, j_2) \textrm{ has } \\ \textrm{no gaps}}}\int \1_{B(T^{i_1}(x), r_n)} (T^{j_1}x) \1_{B(T^{i_2}(x), r_n)} (T^{j_2}x)~d\mu(x)
		\lesssim n^{4-2\beta C_{\mu}-\epsilon_{3}}.
	\end{align*}	
 \end{proposition}
 \begin{proof}
 	We start our proof by introducing two estimates. The first is the following straightforward counting bound that follows from our choice of $\ell$: $$\#\{(i_1, i_2, j_1, j_2):(i_1, i_2, j_1, j_2) \textrm{ has no gaps}\}\lesssim n^{1+o(1)}.$$ The second estimate is the following integral bound:
 	$$\int \1_{B(T^{i_1}(x), r_n)} (T^{j_1}x) \1_{B(T^{i_2}(x), r_n)} (T^{j_2}x)~d\mu(x)\lesssim n^{-\beta \tilde{D}_{\mu}}.$$ Recalling that $i_{1}<j_{1},$ we now see that this bound is a consequence of the following inequalities:
 	\begin{align*}
 		&\int \1_{B(T^{i_1}(x), r_n)} (T^{j_1}x) \1_{B(T^{i_2}(x), r_n)} (T^{j_2}x)~d\mu(x)\\
 		&\quad\leq \int \1_{B(T^{i_1}(x), r_n)} (T^{j_1}x)~d\mu(x)\\
 		&\quad= \int \1_{A_{r_{n}}(j_{1}-i_{1})}(T^{i_{1}}(x))~d\mu(x)\\
 		&\quad=\int \1_{A_{r_{n}}(j_{1}-i_{1})}(x)~d\mu(x)\\
 		&\quad =\mu(A_{r_{n}}(j_{1}-i_{1})) \lesssim n^{-\beta \tilde{D}_{\mu}}.
 	\end{align*} Combining these two estimates yields:
 	\begin{align*}
 		&\sum_{\stackrel{(i_1, i_2, j_1, j_2)}{(i_1, i_2, j_1, j_2) \textrm{ has no gaps}}}\int \1_{B(T^{i_1}(x), r_n)} (T^{j_1}x) \1_{B(T^{i_2}(x), r_n)} (T^{j_2}x)~d\mu(x)\\
 		&\hspace{7cm} \lesssim n^{1-\beta \tilde{D}_{\mu}+o(1)}.
 		\end{align*} The existence of $\epsilon_{3}$ now follows from the inequality $4-2\beta C_{\mu}>1-\beta \tilde{D}_{\mu}$. This inequality holds because of our assumptions $1>\beta(C_{\mu}-\tilde{D}_{\mu})$ and $2>\beta C_{\mu}.$
 \end{proof}
 
  \subsubsection{The one gap case}
 The following proposition considers the case when the quadruple of indices has one gap.
 
 \begin{proposition}
\label{prop:one gap prop}
There exists $\epsilon_{4}>0$ such that for any $s>0$ we have
\begin{align*}
	&\sum_{\substack{(i_1, i_2, j_1, j_2)\\ (i_1, i_2, j_1, j_2) \textrm{ has } \\ \textrm{ one gap}}}\int \1_{B(T^{i_1}(x), r_n)} (T^{j_1}x) \1_{B(T^{i_2}(x), r_n)} (T^{j_2}x)~d\mu(x)
	\lesssim  n^{4-2\beta C_{\mu}-\epsilon_{4}}.
\end{align*}	  
 \end{proposition} 
 
 \begin{proof}
 We begin by stating the counting bound that follows from our choice of function $\ell$:
 $$\#\{(i_1, i_2, j_1, j_2):(i_1, i_2, j_1, j_2) \textrm{ has one gap}\}\lesssim n^{2+o(1)}.$$ It remains to show that if $(i_1, i_2, j_1, j_2)$ has one gap then 
 \begin{equation}
 	\label{eq:Want to show}
 \int \1_{B(T^{i_1}(x), r_n)} (T^{j_1}x) \1_{B(T^{i_2}(x), r_n)} (T^{j_2}x)~d\mu(x)\lesssim n^{2-2\beta C_{\mu}-\epsilon'}
 \end{equation} for some $\epsilon'>0$ that does not depend on $s$. We will do this via a case analysis. We will start by studying Cases B and C. 
 
 \noindent \textbf{Cases B and C. } In these two cases we have either $i_{1}\leq i_{2}\leq j_{1}\leq j_{2}$ or $i_1\leq i_{2}\leq j_{2}\leq j_1.$ Now notice that it follows from these inequalities that if we have one gap, then we must have either $j_{1}-i_{1}\geq \ell(n)$ or $j_{2}-i_{2}\geq \ell(n)$. Let us assume $j_{1}-i_{1}\geq \ell(n).$ The other subcase is handled similarly. Let $\epsilon>0$ and $m$ be defined by \eqref{eq:define m}. Then by Lemma \ref{lem:Riem1}, \eqref{eq:final conditioning}, \eqref{eq:trivial counting bounds}, exponential $4$-mixing on intervals, and the $T$-invariance of $\mu,$ we have 
 \begin{align*}
 	&\int \1_{B(T^{i_1}(x), r_n)} (T^{j_1}x) \1_{B(T^{i_2}(x), r_n)} (T^{j_2}x)~d\mu(x)\\
 	&\quad\leq \sum_{k=0}^{2^{m}-1}\sum_{l=0}^{2^{m}-1}\int   \1_{I_{k}}(x)\cdot \1_{I_{l}}(T^{i_{2}-i_{1}}x)\cdot \1_{B_{k}^{+}}(T^{j_{1}-i_{1}}x)\cdot \1_{B_{l}^{+}}(T^{j_{2}-i_{1}}x) ~d\mu(x)\\
 	&\quad\leq \sum_{k=0}^{2^{m}-1}\sum_{l=0}^{2^{m}-1}\int   \1_{I_{k}}(x)\cdot \1_{I_{l}}(T^{i_{2}-i_{1}}x)\cdot \1_{B_{k}^{+}}(T^{j_{1}-i_{1}}x) ~d\mu(x)\\
 	&\quad= \sum_{k=0}^{2^{m}-1}\int   \1_{I_{k}}(x)\cdot \1_{B_{k}^{+}}(T^{j_{1}-i_{1}}x)~d\mu(x)\\
 	&\quad\lesssim \sum_{k=0}^{2^{m}-1}\mu(I_{k})\mu(B_{k}^{+})+\mathcal{O}(2^{m}\e^{-\theta\ell(n)})=\mathcal{O}( n^{-\beta C_{\mu}+o(1)}).
 \end{align*}
 It follows from the above and our assumption $\beta C_{\mu}<2,$ which implies $-\beta C_{\mu}<2-2\beta C_{\mu},$ that there exists $\epsilon_{1}'$ such that for any $s>0$ we have
 $$\int \1_{B(T^{i_1}(x), r_n)} (T^{j_1}x) \1_{B(T^{i_2}(x), r_n)} (T^{j_2}x)~d\mu(x)\lesssim n^{2-2\beta C_{\mu}-\epsilon_{1}'}.$$
 
\noindent \textbf{Case A. } Recall that in this case we have $i_1\leq j_1\leq i_2\leq j_2$. If $j_{1}-i_{1}\geq \ell(n)$ or $j_{2}-i_{2}\geq \ell(n)$ then  the arguments given for Cases B and C above can be replicated to obtain the required bound. It remains to consider the subcase where $i_{2}-j_{1}\geq \ell(n)$. In this case, the following is a consequence of \eqref{eq:final conditioning}, $T$-invariance of $\mu$, exponential $4$-mixing on intervals and \eqref{eq:counting bounds}. Let $\epsilon>0$ and $m$ be defined by \eqref{eq:define m}, then we have
\begin{align*}
	&\int \1_{B(T^{i_1}(x), r_n)} (T^{j_1}x) \1_{B(T^{i_2}(x), r_n)} (T^{j_2}x)~d\mu(x)\\
	&\quad\leq \sum_{k=0}^{2^{m}-1}\sum_{l=0}^{2^{m}-1}\int   \1_{I_{k}}(x)\cdot  \1_{B_{k}^{+}}(T^{j_{1}-i_{1}}x)\cdot \1_{I_{l}}(T^{i_{2}-i_{1}}x) \cdot \1_{B_{l}^{+}}(T^{j_{2}-i_{1}}x) ~d\mu(x)\\
		&\quad\leq \sum_{k=0}^{2^{m}-1}\sum_{l=0}^{2^{m}-1}\int   \1_{I_{k}}(x)\cdot  \1_{B_{k}^{+}}(T^{j_{1}-i_{1}}x)~d\mu(x)\int \1_{I_{l}}(x) \cdot \1_{B_{l}^{+}}(T^{j_{2}-i_{2}}x) ~d\mu(x)\\
		&\hspace{8cm} +\mathcal{O}(4^{m}\e^{-\theta\ell(n)})\\
		&\quad\leq \sum_{k=0}^{2^{m}-1}\sum_{l=0}^{2^{m}-1}\int   \1_{A_{4r_{n}}(j_{1}-i_{1})}(x)\cdot  \1_{B_{k}^{+}}(T^{j_{1}-i_{1}}x)~d\mu(x) \\
		&\hspace{2cm} \times \int \1_{A_{4r_{n}}(j_{2}-i_{2})}(x)\cdot \1_{B_{l}^{+}}(T^{j_{2}-i_{2}}x) ~d\mu(x)+\mathcal{O}(4^{m}\e^{-\theta\ell(n)})\\
			&\quad\leq n^{2\beta\epsilon}\int   \1_{A_{4r_{n}}(j_{1}-i_{1})}(x) ~d\mu(x)\int \1_{A_{4r_{n}}(j_{2}-i_{2})}(x)~d\mu(x) +\mathcal{O}(4^{m}\e^{-\theta\ell(n)})\\
			&\quad\lesssim n^{2\beta\epsilon -2\beta \tilde{D}_{\mu}}.
		\end{align*}
In the final line we have used that there exists $C>0$ such $\mu(A_{r}(n))\leq Cr^{\tilde{D}_{\mu}}$ for all $r>0$ and $n\in \N$, and the fact our error term decays to zero faster than any polynomial. It is a consequence of our assumption $1>\beta(C_{\mu}-\tilde{D}_{\mu})$ that if we choose $\epsilon$ sufficiently small then $2\beta\epsilon -2\beta \tilde{D}_{\mu}<2-2\beta C_{\mu}$. Therefore it follows from the above that in this case there exists $\epsilon_{2}'>0$ such that for any $s>0$ we have 		
  $$\int \1_{B(T^{i_1}(x), r_n)} (T^{j_1}x) \1_{B(T^{i_2}(x), r_n)} (T^{j_2}x)~d\mu(x)\lesssim n^{2-2\beta C_{\mu}-\epsilon_{2}'}.$$
 
\eqref{eq:Want to show} can be verified for Cases D, E and F via a similar argument and appealing to the symmetry between these cases and Cases A, B and C.

 We've shown that \eqref{eq:Want to show} holds for some $\epsilon'$ for all $s>0$ in all cases, taking the minimum of these $\epsilon'$ yields a uniform estimate and we can conclude that \eqref{eq:Want to show} holds. This completes our proof.
 \end{proof}
 
  \subsubsection{The two gaps case}
 We now consider the case when we have two gaps between our indices. We split our analysis into two cases across the following two propositions. 
 
 We begin by considering the case where there are two gaps between our indices and one of these gaps is between the two smallest indices, i.e., $\gamma_{2}-\gamma_{1}\geq \ell(n)$.
 \begin{proposition}
 	\label{prop:gap first indicies}
Let $(i_1,i_{2},j_{1},j_{2})$ have two gaps and $\gamma_{2}-\gamma_{1}\geq \ell(n)$, then there exists $\epsilon_5>0$ such that for any $s>0$ we have 
\begin{align*}
	&\int \1_{B(T^{i_1}(x), r_n)} (T^{j_1}x) \1_{B(T^{i_2}(x), r_n)} (T^{j_2}x)~d\mu(x)\lesssim n^{1-2\beta C\mu -\epsilon_{5}}.
\end{align*}
 \end{proposition}
 
 \begin{proof}  For concreteness let us suppose we are in one of Case A, Case B or Case C so $\min \{i_{1},i_{2},j_{1},j_{2}\}=i_{1}$. The other cases are handled similarly. Let $\epsilon>0$ and $m$ be defined by \eqref{eq:define m}. By $T$-invariance and \eqref{eq:final conditioning} we have
 	\begin{align*}&\int \1_{B(T^{i_1}(x), r_n)} (T^{j_1}x) \1_{B(T^{i_2}(x), r_n)} (T^{j_2}x)~d\mu(x)\\
 		&\hspace{0.2cm}\leq \sum_{k=0}^{2^m-1}\sum_{l=0}^{2^{m}-1} \int  \1_{I_{k}}(x)\cdot \1_{I_{l}}(T^{i_{2}-i_1}x)\cdot \1_{B_{k}^{+}}(T^{j_{1}-i_1}x)\cdot \1_{B_{l}^{+}}(T^{j_{2}-i_1}x) ~d\mu(x).
 	\end{align*}
 By exponential mixing for BV against $L^{\infty}$, $T$-invariance and the inequality $i_{2}\leq j_2$, we can find $a,b,c\geq 0$ such that either $a=0$ or $b=0$ (depending on which of $i_2$ and $j_1$ is larger) and 
 \begin{align*}
 	&\sum_{k=0}^{2^m-1}\sum_{l=0}^{2^{m}-1} \int  \1_{I_{k}}(x)\cdot \1_{I_{l}}(T^{i_{2}-i_1}x)\cdot \1_{B_{k}^{+}}(T^{j_{1}-i_1}x)\cdot \1_{B_{l}^{+}}(T^{j_{2}-i_1}x) ~d\mu(x)\\
 	&\leq \sum_{k=0}^{2^m-1}\sum_{l=0}^{2^{m}-1} \mu(I_k)\int \1_{I_l}(T^{a}x)\cdot \1_{B_{k}^{+}}(T^{b}x)\cdot \1_{B_{l}^{+}}(T^{c}x) ~d\mu(x)+\mathcal{O}(4^{m}\e^{-\theta \ell(n)}).
 \end{align*}
 
 We now want to bound the integral over the triple product appearing in this summation.  An application of exponential $4$-mixing on intervals allows us to bound this integral from above by a product of two integrals plus some small error. This product over two integrals is of the form where one integral is over one of our indicator functions and the other is over the product of our two remaining indicator functions. The two indicator functions appearing in this latter integral correspond to the successive indices that are close, i.e.\ the $\gamma_{i},\gamma_{i+1}$ satisfying $|\gamma_{i+1}-\gamma_{i}|<\ell(n)$. Consequently after an application of our  exponential $4$-mixing of intervals assumption and $T$-invariance there are three possible forms the resulting bound can take:
 \begin{itemize}
 \item\textbf{ Case 1:} There exist $d,e\geq 0$ such that either $d=0$ or $e=0$ and
 	\begin{align*}
 		&\sum_{k=0}^{2^m-1}\sum_{l=0}^{2^{m}-1} \mu(I_k)\int \1_{I_l}(T^{a}(x))\cdot \1_{B_{k}^{+}}(T^{b}x)\cdot \1_{B_{l}^{+}}(T^{c}x) ~d\mu(x)
		\\
 		&\leq  \sum_{k=0}^{2^m-1}\sum_{l=0}^{2^{m}-1} \mu(I_k)\mu(I_{l})\int  \1_{B_{k}^{+}}(T^{d}(x))\cdot \1_{B_{l}^{+}}(T^{e}(x)) ~d\mu(x)+ \mathcal{O}(4^{m}\e^{-\theta\ell(n)}).
 	\end{align*}
\item \textbf{Case 2:} There exist $d,e\geq 0$ such that either $d=0$ or $e=0$ and
 	\begin{align*}
 		&\sum_{k=0}^{2^m-1}\sum_{l=0}^{2^{m}-1} \mu(I_k)\int \1_{I_l}(T^{a}(x))\cdot \1_{B_{k}^{+}}(T^{b}x)\cdot \1_{B_{l}^{+}}(T^{c}x) ~d\mu(x)
		\\
 		&\leq  \sum_{k=0}^{2^m-1}\sum_{l=0}^{2^{m}-1} \mu(I_k)\mu(B_{k}^{+})\int  \1_{I_{l}}(T^{d}(x))\cdot \1_{B_{l}^{+}}(T^{e}(x)) ~d\mu(x)+\mathcal{O}(4^{m}\e^{-\theta\ell(n)}).
 		\end{align*}
\item \textbf{Case 3:} There exist $d,e\geq 0$ such that either $d=0$ or $e=0$ and
 	\begin{align*}
 			&\sum_{k=0}^{2^m-1}\sum_{l=0}^{2^{m}-1} \mu(I_k)\int \1_{I_l}(T^{a}(x))\cdot \1_{B_{k}^{+}}(T^{b}x)\cdot \1_{B_{l}^{+}}(T^{c}x) ~d\mu(x)
			\\
 		&\leq\sum_{k=0}^{2^m-1}\sum_{l=0}^{2^{m}-1} \mu(I_k)\mu(B_{l}^{+})\int  \1_{B_{k}^{+}}(T^{d}(x))\cdot \1_{I_{l}}(T^{e}(x)) ~d\mu(x)
 		+\mathcal{O}(4^{m}\e^{-\theta\ell(n)}).
 		\end{align*}
 \end{itemize}
 We now consider each one of these cases in turn.
 
 	\textbf{Case 1.}  It follows from the definition of $m$ and our Frostman parameter $\tilde{F}_{\mu}$ that $\mu(I_{l})\lesssim n^{-\beta \tilde{F}_{\mu}}$ for all $0\leq l\leq 2^{m}-1$. Using this bound together with the $T$-invariance of $\mu$, \eqref{eq:counting bounds} and Lemma \ref{lem:Riem1} we see that the following holds:
 	\begin{align*}
 		&\sum_{k=0}^{2^m-1}\sum_{l=0}^{2^{m}-1} \mu(I_k)\mu(I_{l})\int  \1_{B_{k}^{+}}(T^{d}(x))\cdot \1_{B_{l}^{+}}(T^{e}(x)) ~d\mu(x)\\
 		&\quad\lesssim n^{\beta\epsilon-\beta \tilde{F}_{\mu}}\sum_{k=0}^{2^{m}-1}\mu(I_k)\int  \1_{B_{k}^{+}}(x) ~d\mu(x)\\
 		&\quad \lesssim n^{\beta\epsilon-\beta F_{\mu}}\sum_{k=0}^{2^{m}-1}\mu(I_k)\mu(B_{k}^{+}) \\
 		&\quad \lesssim  n^{\beta \epsilon-\beta \tilde{F}_{\mu}-\beta C_{\mu}+o(1)}.
 		\end{align*}
Using our assumption that $\beta(C_{\mu}-\tilde{F}_{\mu})<1,$ which is equivalent to $-\beta \tilde{F}_{\mu}-C_{\mu}<1-2\beta C_{\mu}$, and that $4^{m}\e^{-\theta\ell(n)}$ decays to zero faster than any polynomial, we see that if we choose our original parameter $\epsilon$ to be sufficiently small then there exists $\epsilon_{1}'>0$ such that in this case, for any $s>0$
 $$\int \1_{B(T^{i_1}(x), r_n)} (T^{j_1}x) \1_{B(T^{i_2}(x), r_n)} (T^{j_2}x)~d\mu(x)\lesssim n^{1-2\beta C\mu -\epsilon_{1}'}.$$

 \textbf{Case 2}. In this case it follows from the inequality $i_{2}\leq j_{2}$ that $d=0$. Using Lemma \ref{lem:Riem1}, our early returns assumption, and \eqref{eq:counting bounds} the following holds: 
 \begin{align*}
 	&\sum_{k=0}^{2^m-1}\sum_{l=0}^{2^{m}-1} \mu(I_k)\mu(B_{k}^{+})\int  \1_{I_{l}}(T^{d}(x))\cdot \1_{B_{l}^{+}}(T^{e}(x)) ~d\mu(x)\\
 	&\quad= \sum_{k=0}^{2^{m}-1}\mu(I_k)\mu(B_{k}^{+})\sum_{l=0}^{2^{m}-1}\int  \1_{I_{l}}(x)\cdot \1_{B_{l}^{+}}(T^{e}(x)) ~d\mu(x)\\
 	&\quad\lesssim n^{-\beta C_{\mu}+o(1)}\sum_{l=0}^{2^{m}-1}\int  \1_{I_{l}}(x)\cdot \1_{B_{l}^{+}}(T^{e}(x)) ~d\mu(x)\\
 	&\quad\leq n^{-\beta C_{\mu}+o(1)}\sum_{l=0}^{2^{m}-1}\int  \1_{A_{4r_{n}}(e)}(x)\cdot \1_{B_{l}^{+}}(T^{e}(x)) ~d\mu(x)\\
 		&\quad\lesssim n^{\beta\epsilon-\beta C_{\mu}+o(1)}\int  \1_{A_{4r_{n}}(e)}(x) ~d\mu(x)\\
 		&\quad \lesssim n^{\beta\epsilon -\beta \tilde{D}_{\mu}-\beta C_{\mu}+o(1)}.
 \end{align*}
 Now using that our error terms decays to zero faster than any polynomial, and our assumption $\beta(C_{\mu}-\tilde{D}_{\mu})<1$, which is equivalent to $-\beta \tilde{D}_{\mu}-\beta C_{\mu}<1-2\beta C_{\mu}$, it follows that if we choose our parameter $\epsilon$ to be sufficiently small then  in this case there exists $\epsilon_{2}'>0$ such that for any $s>0$ we have
  $$\int \1_{B(T^{i_1}(x), r_n)} (T^{j_1}x) \1_{B(T^{i_2}(x), r_n)} (T^{j_2}x)~d\mu(x)\lesssim n^{1-2\beta C\mu -\epsilon_{2}'}.$$
 \noindent \textbf{Case 3.} In this case we will use the Frostman bound $\mu(B_{l}^{+})\lesssim n^{-\beta \tilde{F}_{\mu}}.$ Using this bound, \eqref{eq:trivial counting bounds} Lemma \ref{lem:Riem1}, the $T$-invariance of $\mu$, and duplicating the analysis given is Case $1,$ we obtain:
  \begin{align*}
 	&\int \1_{B(T^{i_1}(x), r_n)} (T^{j_1}x) \1_{B(T^{i_2}(x), r_n)} (T^{j_2}x)~d\mu(x)\lesssim n^{-\beta \tilde{F}_{\mu}-\beta C_{\mu}+o(1)}.
 \end{align*} Now using our assumption $\beta(C_{\mu}-\tilde{F}_{\mu})<1,$ which is equivalent to $-\beta \tilde{F}_{\mu}-\beta C_{\mu}<1-2\beta C_{\mu}$, we see that in this case there exists $\epsilon_{3}'>0$ such that for any $s>0$ we have
 $$\int \1_{B(T^{i_1}(x), r_n)} (T^{j_1}x) \1_{B(T^{i_2}(x), r_n)} (T^{j_2}x)~d\mu(x)\lesssim n^{1-2\beta C_{\mu}-\epsilon_{3}'}.$$
 
 Our result now follows from the three cases considered above by taking $\epsilon_{5}=\min \{\epsilon_{1}',\epsilon_{2}',\epsilon_{3}'\}.$ 
 \end{proof}
The following proposition addresses the remaining case when there are two gaps but the first two indices satisfy $\gamma_{2}-\gamma_{1}<\ell(n)$. 
 
 \begin{proposition}
 	\label{prop:close first indicies}
 Let $(i_{1},i_{2},j_{1},j_{2})$ have two gaps and $\gamma_{2}-\gamma_{1}<\ell(n)$, then there exists $\epsilon_{6}>0$ such that for and $s>0$ we have 
 \begin{align*}
 	&\int \1_{B(T^{i_1}(x), r_n)} (T^{j_1}x) \1_{B(T^{i_2}(x), r_n)} (T^{j_2}x)~d\mu(x)\lesssim n^{1-2\beta C_{\mu}-\epsilon_{6}}.
 \end{align*}
 \end{proposition}
 
 \begin{proof}
For concreteness let us again suppose we are in one of Case A, Case B or Case C so $\min \{i_{1},i_{2},j_{1},j_{2}\}=i_{1}$. The other cases are handled similarly. Let $\epsilon>0$ and $m$ be defined by \eqref{eq:define m}. We start with the familiar bound that follows from the $T$-invariance of $\mu$ and \eqref{eq:final conditioning}:
\begin{align*}&\int \1_{B(T^{i_1}(x), r_n)} (T^{j_1}x) \1_{B(T^{i_2}(x), r_n)} (T^{j_2}x)~d\mu(x)\\
&\hspace{0.5cm}\leq \sum_{k=0}^{2^m-1}\sum_{l=0}^{2^{m}-1} \int  \1_{I_{k}}(x)\cdot \1_{I_{l}}(T^{i_{2}-i_1}x)\cdot \1_{B_{k}^{+}}(T^{j_{1}-i_1}x)\cdot \1_{B_{l}^{+}}(T^{j_{2}-i_1}x) ~d\mu(x).
\end{align*}
We emphasise that it follows from our assumptions that $\gamma_{3}-\gamma_{2}\geq \ell(n)$ and $\gamma_{4}-\gamma_{3}\geq \ell(n)$. We will now apply exponential $4$-mixing on intervals to bound the integrals appearing in the summation above  by the product of two integrals plus some error. Using the fact that $i_{2}< j_{2}$, we see that after an application of $4$-mixing we are left with two possibilities:
\begin{itemize}
	\item \textbf{Case 1:} If $j_{1}\geq i_{2}$ then either $i_{1}\leq i_{2}\leq j_{1}\leq j_{2}$ (Case B) or $i_{1}\leq i_{2}\leq j_{2}\leq j_{1}$ (Case C). In which case there exists $a,b\geq 0$ such that $a=0$ (first subcase) or $b=0$ (second subcase), $|a-b|\geq \ell(n),$ and 
	\begin{align*}
		&\sum_{k=0}^{2^m-1}\sum_{l=0}^{2^{m}-1}\int  \1_{I_{k}}(x)\cdot \1_{I_{l}}(T^{i_{2}-i_1}x)\cdot \1_{B_{k}^{+}}(T^{j_{1}-i_1}x)\cdot \1_{B_{l}^{+}}(T^{j_{2}-i_1}x) ~d\mu(x)\\
		&\leq  \sum_{k=0}^{2^m-1}\sum_{l=0}^{2^{m}-1}\int \1_{I_{k}}(x)\cdot \1_{I_{l}}(T^{i_{2}-i_1}x)~d\mu(x)\int \1_{B_{k}^{+}}(T^{a}x)\cdot \1_{B_{l}^{+}}(T^{b}x) ~d\mu(x)\\
		&\quad + \mathcal{O}(4^{m}\e^{-\theta\ell(n)}).
		\end{align*}
			\item \textbf{Case 2:} If $j_{1}< i_{2}$ then we must be in Case A where $i_{1}\leq j_{1}\leq i_{2}\leq j_{2}$ since we always have $i_{2}\leq  j_{2}$. We also have the bound $j_{2}-i_{2}\geq \ell(n)$ by our underling assumptions.  In this case, an application of exponential $4$-mixing on intervals yields
		\begin{align*}
			&\sum_{k=0}^{2^m-1}\sum_{l=0}^{2^{m}-1}\int   \1_{I_{k}}(x)\cdot\1_{I_{l}}(T^{i_{2}-i_1}x)\cdot \1_{B_{k}^{+}}(T^{j_{1}-i_1}x)\cdot \1_{B_{l}^{+}}(T^{j_{2}-i_1}x) ~d\mu(x)\\
			&\leq \sum_{k=0}^{2^m-1}\sum_{l=0}^{2^{m}-1}  \int  \1_{I_{k}}(x)\cdot \1_{B_{k}^{+}}(T^{j_{1}-i_1}x)~d\mu(x)\int \1_{I_{l}}(x) \cdot \1_{B_{l}^{+}}(T^{j_{2}-i_{2}}x) ~d\mu(x)\\
			&\quad+ \mathcal{O}(4^{m}\e^{-\theta\ell(n)}).
		\end{align*}
\end{itemize}
We consider these two cases in turn. Because the error terms appearing in these two cases decay to zero faster than any polynomial they can be ignored in what follows. As such, we will just focus on bounding the integral terms.

\noindent \textbf{Case 1.} In this case we use our mixing assumption in the second integral and then the Cauchy-Schwarz inequality to obtain: 
\begin{align*}
	&\sum_{k=0}^{2^{m}-1}\sum_{l=0}^{2^{m}-1}\int \1_{I_{k}}(x)\cdot \1_{I_{l}}(T^{i_{2}-i_1}x)~d\mu(x)\int \1_{B_{k}^{+}}(T^{a}x)\cdot \1_{B_{l}^{+}}(T^{b}x) ~d\mu(x)\\
	&\leq \sum_{k=0}^{2^{m}-1}\sum_{l=0}^{2^{m}-1}\mu(B_{k}^{+})\mu(B_{l}^{+})\int \1_{I_{k}}(x)\cdot \1_{I_{l}}(T^{i_{2}-i_1}x)~d\mu(x)+ \mathcal{O}(4^{m}\e^{-\theta\ell(n)})\\
	&\leq \left(\sum_{k=0}^{2^{m}-1}\sum_{l=0}^{2^{m}-1}\mu(B_{k}^{+})^{2}\int \1_{I_{k}}(x)\cdot \1_{I_{l}}(T^{i_{2}-i_1}x)~d\mu(x)\right)^{1/2}\\
	&\quad \times \left(\sum_{k=0}^{2^{m}-1}\sum_{l=0}^{2^{m}-1}\mu(B_{l}^{+})^{2}\int \1_{I_{k}}(x)\cdot \1_{I_{l}}(T^{i_{2}-i_1}x)~d\mu(x)\right)^{1/2} + \mathcal{O}(4^{m}\e^{-\theta\ell(n)}).
\end{align*} 
Now using the Frostman bounds $\mu(B_{k}^{+})\lesssim n^{-\beta \tilde{F}_{\mu}},$ $\mu(B_{l}^{+})\lesssim n^{-\beta \tilde{F}_{\mu}}$ together with \eqref{eq:trivial counting bounds}, and Lemma \ref{lem:Riem1}, it follows from the above that 
\begin{align*}
	&\sum_{k=0}^{2^{m}-1}\sum_{l=0}^{2^{m}-1}\int \1_{I_{k}}(x)\cdot \1_{I_{l}}(T^{i_{2}-i_1}x)~d\mu(x)\int \1_{B_{k}^{+}}(T^{a}x)\cdot \1_{B_{l}^{+}}(T^{b}x) ~d\mu(x)\\
	&\quad\leq  \left(n^{-\beta \tilde{F}_{\mu}}\sum_{k=0}^{2^{m}-1}\mu(I_{k})\mu(B_{k}^{+})\right)^{1/2}	\left(n^{-\beta \tilde{F}_{\mu}}\sum_{l=0}^{2^{m}-1}\mu(I_{l})\mu(B_{l}^{+})\right)^{1/2}\\
	&\qquad+ \mathcal{O}(4^{m}\e^{-\theta\ell(n)})\\
	&\quad\lesssim n^{-\beta \tilde{F}_{\mu}-\beta C_{\mu}+o(1)}+ \mathcal{O}(4^{m}\e^{-\theta\ell(n)})\\
	&\quad \lesssim n^{-\beta \tilde{F}_{\mu}-\beta C_{\mu}+o(1)}.
\end{align*}
In the final line we have used that our error term decays to zero faster than any polynomial.
Now using our assumption $\beta(C_{\mu}-\tilde{F}_{\mu})<1,$ which is equivalent to $-\beta \tilde{F}_{\mu}-\beta C_{\mu}< 1- 2\beta C_{\mu}$, it follows that in this case there exists $\epsilon_{1}'>0$ such that for all $s>0$ we have
\begin{equation}
	\label{eq:Case 1 eps}
\int \1_{B(T^{i_1}(x), r_n)} (T^{j_1}x) \1_{B(T^{i_2}(x), r_n)} (T^{j_2}x)~d\mu(x)\lesssim n^{1-2\beta C_{\mu}-\epsilon_{1}'}.
\end{equation}

\noindent \textbf{Case 2.} Using our mixing assumptions in the second integral, Lemma \ref{lem:Riem1}, our early return assumption, and \eqref{eq:counting bounds} we obtain 
\begin{align*}
	&\sum_{k=0}^{2^{m}-1}\sum_{l=0}^{2^{m}-1}\int  \1_{I_{k}}(x)\cdot \1_{B_{k}^{+}}(T^{j_{1}-i_1}x)~d\mu(x)\int \1_{I_{l}}(x) \cdot \1_{B_{l}^{+}}(T^{j_{2}-i_{2}}x) ~d\mu(x)\\
	&\leq \sum_{k=0}^{2^{m}-1}\sum_{l=0}^{2^{m}-1}\mu(I_{l})\mu(B_{l}^{+})\int  \1_{I_{k}}(x)\cdot \1_{B_{k}^{+}}(T^{j_{1}-i_1}x)~d\mu(x)+
	\mathcal{O}(4^{m}\e^{-\theta\ell(n)})\\
		&\lesssim n^{-\beta C_{\mu}+o(1)}\sum_{k=0}^{2^{m}-1}\int  \1_{I_{k}}(x)\cdot \1_{B_{k}^{+}}(T^{j_{1}-i_1}x)~d\mu(x)+
	\mathcal{O}(4^{m}\e^{-\theta\ell(n)})\\
	&\lesssim n^{-\beta C_{\mu}+o(1)}\sum_{k=0}^{2^{m}-1}\int  \1_{A_{4r_{n}}(j_{1}-i_{1})}(x)\cdot \1_{B_{k}^{+}}(T^{j_{1}-i_1}x)~d\mu(x)+
	\mathcal{O}(4^{m}\e^{-\theta\ell(n)})\\
	&\lesssim n^{\beta \epsilon-\beta C_{\mu}+o(1)}\int  \1_{A_{4r_{n}}(j_{1}-i_{1})}(x)~d\mu(x)+
	\mathcal{O}(4^{m}\e^{-\theta\ell(n)})\\
	&\lesssim n^{\beta \epsilon-\beta \tilde{D}_{\mu}-\beta C_{\mu}+o(1)}+
	\mathcal{O}(4^{m}\e^{-\theta\ell(n)})\\
	&\lesssim n^{\beta \epsilon-\beta \tilde{D}_{\mu}-\beta C_{\mu}+o(1)}.
\end{align*} 
In the final line we used that our error term decays to zero faster than any polynomial. It follows now from our assumption $\beta(C_{\mu}-\tilde{D}_{\mu})<1$, which is equivalent to $-\beta \tilde{D}_{\mu}-\beta C_{\mu}<1-2\beta C_{\mu}$, that if $\epsilon$ is chosen to be sufficiently small, then there exists $\epsilon_{2}'>0$ such that for all $s>0$ we have
\begin{equation}
	\label{eq:Case 2 eps}
\int \1_{B(T^{i_1}(x), r_n)} (T^{j_1}x) \1_{B(T^{i_2}(x), r_n)} (T^{j_2}x)~d\mu(x)\lesssim n^{1-2\beta C_{\mu}-\epsilon_{2}'}.
\end{equation}

Our result now follows from \eqref{eq:Case 1 eps} and \eqref{eq:Case 2 eps} by choosing $\epsilon$ sufficiently small and letting $\epsilon_{5}=\min \{\epsilon_{1}',\epsilon_{2}'\}.$
 \end{proof}
 Combining Propositions \ref{prop:gap first indicies} and \ref{prop:close first indicies} we have the following result.
 
 \begin{proposition}
 	\label{prop:two big gaps prop}
 	There exists $\epsilon_{7}>0$ such that for all $s>0$ we have $$\sum_{\substack{(i_1, i_2, j_1, j_2)\\ (i_1, i_2, j_1, j_2) \textrm{ has } \\ \textrm{ two gaps}}}\int \1_{B(T^{i_1}(x), r_n)} (T^{j_1}x) \1_{B(T^{i_2}(x), r_n)} (T^{j_2}x)~d\mu(x)\lesssim n^{4-2\beta C_{\mu}-\epsilon_{7}}.$$
 \end{proposition}
 \begin{proof}
 	This result follows immediately from  Propositions \ref{prop:gap first indicies}, \ref{prop:close first indicies} and the following straightforward upper bound that follows from our choice of function $\ell$ (recall Definition \ref{def:ell}):
 	$$\#\{(i_1, i_2, j_1, j_2):(i_1, i_2, j_1, j_2) \textrm{ has two gaps}\}\lesssim n^{3+o(1)}.$$
 \end{proof}
 
  \subsubsection{The not well separated conclusion}
 Now combining Propositions \ref{prop:no gap prop}, \ref{prop:one gap prop} and \ref{prop:two big gaps prop} we can conclude the following statement.
 
\begin{proposition}
	\label{prop:not well separated}
There exists $\epsilon_{8}>0$ such that for all $s>0$ we have
$$\sum_{\substack{(i_1, i_2, j_1, j_2) \\ (i_1, i_2, j_1, j_2) \textrm{ is not }\\ \textrm{ well separated}}}\int \1_{B(T^{i_1}(x), r_n)} (T^{j_1}x) \1_{B(T^{i_2}(x), r_n)} (T^{j_2}x)~d\mu(x)\lesssim n^{4-2\beta C_{\mu}-\epsilon_{8}}.$$
\end{proposition}

\subsection{Proof of Theorem \ref{mainthm:1orb}}
Equipped with the expectation estimates established in the previous section we are now in a position to prove Theorem \ref{mainthm:1orb}. Our proof is similar to the proof of Theorem~\ref{thm:2orbgen}.

\begin{proof}[Proof of Theorem \ref{mainthm:1orb}]
Our proof relies upon showing that there exists $\delta>0$ such that for any $s>0$ we have the following second moment estimate:
\begin{equation}
	\label{eq:delta decay}
\E\left(\left(\frac{S_{n}}{\E(S_{n})}-1\right)^{2}\right)=	\frac{\E(S_{n}^{2})-\E(S_{n})^{2}}{\E(S_{n})^{2}}\lesssim n^{-\delta}.
\end{equation}
Once \eqref{eq:delta decay} is established we can apply the same arguments as used in the proof of Theorem \ref{thm:2orbgen} to complete our proof. In particular, we can apply Markov's inequality and the Borel-Cantelli lemma to show that for a fixed $s>0$ we have the desired convergence for $\mu$-almost every $x$ along some sequence $(n^{K})_{n}$ where $K$ does not depend on $s$. We then use an approximation argument and our assumption that $\mu$ has continuous mean scaling to upgrade this statement to show that for $\mu$-almost every $x,$ for any $s>0$ we have the desired convergence along the integers. The only minor difference between the proofs is that in the proof of Theorem \ref{thm:2orbgen} we use the simple formula for $\E(S_n)$ provided by \eqref{eq:Expectation lower bound two orbit}. Whereas in the one orbit case we have to use the following more complicated formula for $\E(S_n)$ which is a consequence of Proposition \ref{prop:expectation}:
 $$\E(S_n)=n^{2}\int \mu(B(x,r_n))~d\mu(x)+\mathcal{O}(n^{2-\beta C_{\mu}-\epsilon_{1}}).$$ Because the error term is of a lower order than the integral term this causes no significant issues.

We now turn our attention to establishing \eqref{eq:delta decay}. By Proposition \ref{prop:expectation} and \eqref{eq:integral asymptotics}, to show that \eqref{eq:delta decay} holds for some $\delta>0$ for all $s>0$, it is sufficient to show that 
\begin{equation}
	\label{eq:delta decay2}
\frac{\E(S_{n}^{2})-\E(S_{n})^{2}}{n^{4}\left(\int\mu(B(x,r_{n}))~d\mu(x)\right)^{2}}\lesssim n^{-\delta}
\end{equation} for some $\delta>0$ for all $s>0$. Recall that $$\E(S_{n}^{2})=4\sum_{i_1< j_1\in [0, n)} \sum_{i_2< j_2\in [0, n)} \int \1_{B(T^{i_1}(x), r_n)} (T^{j_1}x) \1_{B(T^{i_2}(x), r_n)}(T^{j_{2}}x)\, d\mu(x).$$
Combining Propositions \ref{prop:Well separated} and \ref{prop:not well separated} with \eqref{eq:integral asymptotics}, and using the fact that 
\begin{align*}
&\#\{(i_1, i_2, j_1, j_2):(i_1, i_2, j_1, j_2) \textrm{ is well separated}\, i_{1}<j_{1} \textrm{ and } i_{2}<j_{2}\}\\
&\hspace{8cm}=\frac{n^{4}}{4}+\mathcal{O}(n^{3+o(1)}),
\end{align*} for any $s>0$ we have
\begin{align}
	\label{eq:E(S_n^2) bound}
&\left|\E(S_{n}^{2})-n^{4}\left(\int\mu(B(x,r_{n}))~d\mu(x)\right)^{2}\right|\nonumber \\
&\leq \left|4\sum_{\substack{(i_1, i_2, j_1, j_2)\\ (i_1, i_2, j_1, j_2) \textrm{ is } \\ \textrm{well separated}}}\int \1_{B(T^{i_1}(x), r_n)} (T^{j_1}x) \1_{B(T^{i_2}(x), r_n)} (T^{j_2}x)~d\mu(x) \right.\nonumber\\
&\hspace{60mm} \left.-n^{4}\left(\int\mu(B(x,r_{n}))~d\mu(x)\right)^{2}\right|\nonumber\nonumber\\
&+\left|4\sum_{\substack{(i_1, i_2, j_1, j_2) \\ (i_1, i_2, j_1, j_2) \textrm{ is not } \\ \textrm{well separated}}}\int \1_{B(T^{i_1}(x), r_n)} (T^{j_1}x) \1_{B(T^{i_2}(x), r_n)} (T^{j_2}x)~d\mu(x)\right|\nonumber\\
&\lesssim n^{4-2\beta C_{\mu}-\epsilon_{2}}+n^{4-2\beta C_{\mu}-\epsilon_{8}}+n^{3-2\beta C_{\mu}+o(1)}\nonumber\\
&\lesssim n^{4-2\beta C_{\mu}-\min\{\epsilon_{3},\epsilon_{8}\}}.
\end{align}
In the last line we have assumed that $\min \{\epsilon_{3},\epsilon_{8}\}<1,$ which we may freely do without loss of generality.
Moreover by Proposition \ref{prop:expectation}, \eqref{eq:integral asymptotics} and the well known equality $x^{2}-y^{2}= (x-y)^{2}+2(x-y)y$ for $x,y\in \R,$ for any $s>0$ we have
\begin{align}
	\label{eq:E(S_n)^2 bound}
\left|\E(S_{n})^{2}-n^{4}\left(\int \mu(B(x,r_{n}))~d\mu(x)\right)^{2}\right|& \lesssim n^{4-2\beta C_{\mu}-2\epsilon_{1}}+n^{4-2\beta C_{\mu}-\epsilon_{1}+o(1)}\nonumber\\
&\lesssim n^{4-2\beta C_{\mu}-\epsilon_{1}+o(1)}.
\end{align} Combining the triangle inequality, \eqref{eq:E(S_n^2) bound}, \eqref{eq:E(S_n)^2 bound} and \eqref{eq:integral asymptotics}, for any $s>0$ we have that 
\begin{equation*}
\frac{\E(S_{n}^{2})-\E(S_{n})^{2}}{n^{4}\left(\int\mu(B(x,r_{n}))~d\mu(x)\right)^{2}}\lesssim \frac{n^{4-2\beta C_{\mu}-\min\{\epsilon_{1},\epsilon_{3},\epsilon_{8}\}+o(1)}}{n^{4-2\beta C_{\mu}+o(1)}}\lesssim n^{-\min\{\epsilon_{1},\epsilon_{3},\epsilon_{8}\}/2}.
\end{equation*}
Taking $\delta=\min\{\epsilon_{1},\epsilon_{3},\epsilon_{8}\}/2$ we see that \eqref{eq:delta decay2} holds and so \eqref{eq:delta decay} also holds. This completes our proof.
\end{proof}

\section{Applications}

\label{sec:apps}

Here we will first apply our results to acips for Gibbs-Markov maps and the Gauss map, which is the content of Corollaries~\ref{cor:2orbacip} and \ref{cor:1orbacip}.  We then give examples of non-acip cases where our results apply, and finally an (acip) example of a slow mixing system where the conclusion of Corollary \ref{cor:1orbacip} does not hold.

\subsection{Acips for Gauss and Gibbs-Markov maps}

\label{ssec:acipapps}

Here we will prove Corollaries~\ref{cor:2orbacip} and \ref{cor:1orbacip}.  

\begin{lemma} If $\mu_1, \mu_2$ are acips on $[0, 1]$ with densities $\rho_1, \rho_2$ in BV, then 
$$\int\mu_1(B(x, r))~d\mu_2(x) \sim 2r\int(\rho_1\rho_2)(x)~dx.$$
\label{lem:acipdense}
\end{lemma}

\begin{proof}
	Let $m\in \N$ be arbitrary. Then for $r<1/2m$ we have
	\begin{align*}
		\int \mu_{1}(B(x,r))~d\mu_{2}(x)&=\sum_{l=0}^{m-1}\int_{[\frac{l}{m}+r,\frac{l+1}{m}-r]}\mu_{1}(B(x,r))~d\mu_{2}(x)\\
		&\quad+\sum_{l=0}^{m-1}\int_{[\frac{l}{m},\frac{l}{m}+r]\cup [\frac{l+1}{m}-r,\frac{l+1}{m}]}\mu_{1}(B(x,r))~d\mu_{2}(x).
	\end{align*}
For the second term we have	
$$\sum_{l=0}^{m-1}\int_{[\frac{l}{m},\frac{l}{m}+r]\cup (\frac{l+1}{m}-r,\frac{l+1}{m}]}\mu_{1}(B(x,r))~d\mu_{2}(x)\leq 2r^{2}m\|\rho_{1}\|_{\infty}\|\rho_{2}\|_{\infty}.$$ In our analysis of the first term we will use $\textrm{Var}_{a}^{b}(\rho_{1})$ to denote the total variation of $\rho_{1}$ on a closed interval $[a,b].$ In this case we have
\begin{align*}
	&\sum_{l=0}^{m-1}\int_{[\frac{l}{m}+r,\frac{l+1}{m}-r]}\mu_{1}(B(x,r))~d\mu_{2}(x)\\
	&\quad=  \sum_{l=0}^{m-1}\int_{[\frac{l}{m}+r,\frac{l+1}{m}-r]}2r\rho_{1}(x)~d\mu_{2}(x)\\
	&\qquad +\sum_{l=0}^{m-1}\int_{[\frac{l}{m}+r,\frac{l+1}{m}-r]}\int_{x-r}^{x+r}\rho_{1}(y)-\rho_{1}(x)~dy~d\mu_{2}(x)\\
	&\quad\leq 2r \int\rho_{1}(x)\rho_{2}(x)~dx+ \sum_{l=0}^{m-1}\int_{[\frac{l}{m}+r,\frac{l+1}{m}-r]}2r\cdot \textrm{Var}_{\frac{l}{m}}^{\frac{l+1}{m}}(\rho_{1})~d\mu_{2}(x)\\
	& \quad\leq 2r \int\rho_{1}(x)\rho_{2}(x)~dx+ \frac{2r}{m}\|\rho_{2}\|_{\infty}\sum_{l=0}^{m-1} \textrm{Var}_{\frac{l}{m}}^{\frac{l+1}{m}}(\rho_{1})\\	
	&\quad= 2r \left(\int\rho_{1}(x)\rho_{2}(x)~dx+ \frac{1}{m}\|\rho_{2}\|_{\infty}\textrm{Var}_{0}^{1}(\rho_{1})\right).
\end{align*}
In the final line we have used that the total variation is additive on closed intervals. Combining the estimates above we have
\begin{align*}
&\int \mu_{1}(B(x,r))~d\mu_{2}(x)\\
&\hspace{1.5cm} \leq 2r\left(\int\rho_{1}(x)\rho_{2}(x)~dx+ \frac{1}{m}\|\rho_{2}\|_{\infty}\textrm{Var}_{0}^{1}(\rho_{1})+rm\|\rho_{1}\|_{\infty}\|\rho_{2}\|_{\infty}\right).
\end{align*} We see that all of the terms in the large bracket apart from $ \int\rho_1\rho_2~dx$ can either be made arbitrarily small by choosing $m$ sufficiently large, or tend to zero as $r\to 0$. This implies
$$\limsup_{r\to 0}\frac{\int\mu_1(B(x, r))~d\mu_2(x) }{2r\int(\rho_1\rho_2)(x)~dx}\leq 1.$$  The corresponding lower bound for the liminf follows from a similar argument.
\end{proof}

We are now in a position to prove the corollaries.

\begin{proof}[Proof of Corollary~\ref{cor:2orbacip}]
Since our maps are Gibbs-Markov or the Gauss map, the acips $\mu_1$ and $\mu_2$ have densities $\rho_1$ and $\rho_2$ in BV. Moreover, $\rho_{1}$ and $\rho_{2}$ are both bounded away from 0 and bounded above.  In particular this means that $C_{\mu_1, \mu_2}=1$ and $\mu_1$ and $\mu_2$ have continuous mean scaling.  Furthermore, $(X, T_i, \mu_i)$ have exponential mixing for
    $BV$ against $L^\infty$ for $i=1, 2$, see for example \cite{Ryc83} for the Gibbs-Markov case and \cite[Section 2.6]{BruDemTod18} for that case, done in more details, and for the Gauss map case.  Therefore this result follows from Theorem~\ref{mainthm:2orb} and Lemma~\ref{lem:acipdense}.
  \end{proof}
  
\begin{proof}[Proof of Corollary~\ref{cor:1orbacip}]
This proof is almost the same as that for Corollary~\ref{cor:2orbacip}. The correlation dimension exists and $\mu$ has continuous mean scaling by the arguments above. By \cite[Lemma~4.16]{Zha24} we satisfy the required exponential 4-mixing on intervals assumption. The Frostman dimension is $1$ since the density is bounded above. Moreover, by \cite[Lemma~3.4]{HolNicTor12} the early return property holds with exponent $D_\mu=1$. Our result now follows from Theorem \ref{mainthm:1orb} and Lemma \ref{lem:acipdense}.
\end{proof}

\subsection{The continuous mean scaling property}

Since our applications require the continuous mean scaling property, which is not always easy to check in non-acip settings, we give an alternative criterion for the single measure case in the following lemma.

\begin{lemma}
Suppose that $\mu$ has the property that there exists $r_{0}>0,$ $D>1$ and $C>1$ such that for any $x\in \textrm{supp}(\mu)$ and $r\in (0,r_{0})$ we have $$\mu(B(x,Dr))\geq C\mu(B(x,r)).$$ Then $\mu$ has the continuous mean scaling property.
\label{lem:backdouble}
\end{lemma}

Note that we can iterate the relation above: so long as $D^{n-1}r<r_0$,
\begin{equation}
\mu(B(x,D^nr))\geq C^n\mu(B(x,r)).\label{eq:backdoubleiter}
\end{equation}

As we will see, this condition is satisfied by many dynamically defined measures, but the fact that it holds for  acips with density bounded above and below and Ahlfors regular measures is immediate.

\begin{proof}
Fix $\mu$ satisfying our assumptions and $s,\epsilon,\beta >0.$ Then for any $\delta>0$ we have
\begin{align*}
0& \leq \frac{\int \mu(B(x,(s+\delta)/n^{\beta}))~d\mu(x)}{\int \mu(B(x,s/n^{\beta}))~d\mu(x)}-1\\
&=\frac{\int \mu(B(x,(s+\delta)/n^{\beta})\setminus B(x,s/n^{\beta}) )~d\mu(x)}{\int \mu(B(x,s/n^{\beta}))~d\mu(x)}\\
&\leq \frac{\int \mu(B(x+s/n^{\beta},\delta/n^{\beta}))~d\mu(x)}{\int \mu(B(x,s/n^{\beta}))~d\mu(x)}\\
&\quad + \frac{\int \mu(B(x-s/n^{\beta},\delta/n^{\beta}) )~d\mu(x)}{\int \mu(B(x,s/n^{\beta}))~d\mu(x)}.
\end{align*}
We will now show that we can choose $\delta_{0}>0$ such that when $0<\delta\leq \delta_{0}$ then $$\frac{\int \mu(B(x+s/n^{\beta},\delta/n^{\beta}))~d\mu(x)}{\int \mu(B(x,s/n^{\beta}))~d\mu(x)}<\epsilon/2$$ for all $n\in \N.$ The argument we present can easily be adapted to find $\delta_{0}>0$ for which 
$$\frac{\int \mu(B(x-s/n^{\beta},\delta/n^{\beta}) )~d\mu(x)}{\int \mu(B(x,s/n^{\beta}))~d\mu(x)}<\epsilon/2$$ for all $0<\delta\leq \delta_{0}$ and $n\in \N$, and to show that for any $s,\beta,\epsilon>0$ there exists $\delta_{0}>0$ such that
$$\left| \frac{\int \mu(B(x,(s-\delta)/n^{\beta}))~d\mu(x)}{\int \mu(B(x,s/n^{\beta}))~d\mu(x)}-1\right|<\epsilon$$ for all $0<\delta\leq \delta_{0}$ for all $n\in \N$. Thus $\mu$ will have continuous mean scaling.

Given $\delta>0$ and $n\in \N$ let $L\in \N$ be such that $$2^{-L-1}< \frac{\delta}{n^{\beta}}\leq 2^{-L}$$ and $w\in \mathbb{N}_0$ be such that $$\frac{w}{2^{L}}\leq \frac{s}{n^{\beta}}\leq \frac{w+1}{2^{L}}.$$ Letting $I_j= \left[\frac j{2^L}, \frac{j+1}{2^L}\right)$ for $0\leq j\leq 2^{L-1}$, we have 
\begin{align*}\int \mu(B(x+s/n^{\beta},\delta/n^{\beta}))~d\mu(x)&\leq \sum_{j=0}^{2^{L}-1}\mu(I_{j})\sum_{p=-3}^{3}\mu(I_{j+w+p})\\
	&\leq \sum_{p=-3}^{3}\sum_{j=0}^{2^{L}-1}\mu(I_{j})\mu(I_{j+w+p})\\
	& \leq \sum_{p=-3}^{3}\left(\sum_{j=0}^{2^{L}-1}\mu(I_{j})^{2}\right)^{1/2}\left(\sum_{j=0}^{2^{L}-1}\mu(I_{j+w+p})^{2}\right)^{1/2}\\
	&\leq 7 \sum_{j=0}^{2^{L}-1}\mu(I_{j})^{2}.
	\end{align*} Where we have used the Cauchy-Schwarz inequality in the penultimate line. For any $P>0$ and interval $I$ we let $P\cdot I$ be the interval with the same centre as $I$ expanded by a factor of $P$. Since $s$ is fixed, for any $P>1$ we can choose $\delta_{0}>0$ to be sufficiently small so that for any $0< \delta\leq \delta_{0}$ and $n\in \N$ we have 
	$$\int \mu(B(x,s/n^{\beta}))~d\mu(x)\geq \sum_{j=0}^{2^{L}-1}\mu(I_{j})\mu(P\cdot I_j).$$
	Combining the estimates above, we have shown that for any $P>1$ we can choose $\delta_{0}>0$ so that for any $0<\delta\leq \delta_{0}$ and $n\in \N$ we have
	\begin{equation}
		\label{eq:P growth}
	\frac{\int \mu(B(x+s/n^{\beta},\delta/n^{\beta}))~d\mu(x)}{\int \mu(B(x,s/n^{\beta}))~d\mu(x)}\leq \frac{7\sum_{j=0}^{2^{L}-1}\mu(I_{j})^{2}}{\sum_{j=0}^{2^{L}-1}\mu(I_{j})\mu(P\cdot I_j)}.
		\end{equation}It follows now from \eqref{eq:backdoubleiter} that for any $\epsilon>0,$ we can choose $\delta_{0}>0$ and $P>1$ so that \eqref{eq:P growth} holds for all $0<\delta\leq \delta_{0}$ and $n\in \N,$ and for any interval $I$ of at length at most $2\delta_{0}$ satisfying $\mu(I)>0,$ we have $\mu(I)<\frac{\epsilon}{2}\cdot \mu(P\cdot I).$ Using this final inequality in \eqref{eq:P growth} implies $$\frac{\int \mu(B(x+s/n^{\beta},\delta/n^{\beta}))~d\mu(x)}{\int \mu(B(x,s/n^{\beta}))~d\mu(x)}<\epsilon/2$$ for all $0<\delta\leq \delta_{0}$ for all $n\in \N.$ Which was what we wanted to show. 	
\end{proof}

\subsection{Cookie cutters: the general case}  
\label{ssec:cookiegeneral}
Here we outline how we can apply our results to cookie cutters and their associated equilibrium states.  

Recall Definition \ref{def:Gibbs-Markov} and the discussion on cookie cutters after Corollary \ref{cor:2orbacip}. We begin by remarking that there is some minimal gap $G>0$ such that $\dist(P, P') \ge G$ for distinct $P, P'\in \P$. Now let $\P_n$ denote the set of \emph{$n$-cylinders}, i.e., maximal intervals $P$ such that for each $i=0, \ldots, n-1$,  $\tilde T^iP\subset P_i$ for some $P_i\in \P$.  The bounded distortion property, along with the regularity and expansion properties, implies that $\frac{D\tilde T^n x}{D\tilde T^ny}$ for $x, y\in P\in \P_n$ is uniformly bounded from above and below.  Hence there is $G'>0$ such that if $P, P'\in \P_n$ are distinct and $P\subset P''\in \P_{n-1}$, then $\dist(P, P') \ge G'|P''|$.  Note that by the bounded distortion property there also exists $b_T>0$ such that $|P|\ge b_T|P''|$.

  We next briefly explain the thermodynamic formalism associated to cookie cutters, see for example \cite[Chapter 5]{PrzUrb10} for proofs of the claims made below.

We note that for $\phi:\tilde X\to \R$ H\"older continuous there is a unique equilibrium state $\mu_\phi$, that is $\mu_{\phi}$ is the unique $T$-invariant measure satisfying
$$h(\mu_\phi)+ \int\phi~d\mu_\phi= P(\phi):= \sup\left\{h(\nu) + \int\phi~d\nu: \nu \in \M\right\}$$
where $\M$ is the space of $T$-invariant probability measures.   Moreover, there is a H\"older continuous $\rho$ (bounded away from zero and infinity) such that for $\psi= \phi+\log \rho-\log \rho\circ T -P(\phi)\le -c_\phi<0$, for some $c_\phi>0$, $\mu_\phi$ is also an equilibrium state for $\psi$, but it is furthermore \emph{$\psi$-conformal}: this means that if $T:A\to T(A)$ is bijective on a measurable set $A$, then 
$$\mu_\phi(T(A)) = \int_A\e^{-\psi}~d\mu_\phi.$$
Iterating this we obtain that if $T^n:A\to T^n(A)$ is bijective, then
$$\mu_\phi(T^n(A)) = \int_A\e^{-S_n\psi}~d\mu_\phi$$
where $S_n\psi(x) := \sum_{k=0}^{n-1}\psi(T^kx)$.  

Noting that the above means that if $\phi$ is $\alpha$-H\"older then so is $\psi$, we will immediately obtain some additional regularity for $S_n\psi$ which is provided by the following lemma. 

\begin{lemma}
If $\psi$ is $\alpha$-H\"older then there exists $K\ge 1$ such that 
for any $n\in\N$, for $x, y\in P\in \P_n$, 
$$|S_n\psi(x) - S_n\psi(y)|\le K|\tilde T^nx- \tilde T^ny|^\alpha.$$
\label{lem:psireg}
\end{lemma}

\begin{proof}
Suppose $|\psi(x) -\psi(u)|\le \tilde K|x-y|^\alpha$.  Then recalling that $|D\tilde T|\ge \lambda>1$,
\begin{align*}
|S_n\psi(x) - S_n\psi(y)| & \le \tilde K\sum_{k=0}^{n-1}|\tilde T^k(x)- \tilde T^k(y)|^\alpha \le \tilde K|\tilde T^nx- \tilde T^ny|^\alpha \sum_{k=0}^{n-1}\frac1{\lambda^{\alpha k}}\\
&\le K|\tilde T^nx- \tilde T^ny|
\end{align*}
for $K= \tilde K/(1-\lambda^{-\alpha})$.
\end{proof}

We will now prove that the assumptions appearing in our main theorems hold for equilibrium states associated to H\"{o}lder continuous potentials. First note that by \cite{Ryc83} and \cite[Lemma~4.16]{Zha24}, the mixing properties required in Theorems~\ref{mainthm:2orb}, \ref{mainthm:1orb} and \ref{thm:2orbgen} hold. 
As in \cite{PesWei97}, $C_\mu$ exists for equilibrium states associated to cookie cutters, see also \cite[Chapter 9]{PrzUrb10}.  Note that $D_q(\mu)$, the $q$-Renyi dimension was first rigorously shown to exist for equilibrium states associated to cookie cutters in \cite{Ran89}, and the proof that $D_2(\mu)= C_\mu$ was given in the other references.

\begin{lemma}
Let $\mu$ be an equilibrium state with Frostman exponent $F_{\mu}$. For any $\tilde{F}_\mu<F_\mu$, there exists $C>0$ such that
$\mu(A_r(n))\le C r^{\tilde{F}_\mu}$ for all $r>0$ and $n\in \N$. Thus the early return exponent satisfies $D_{\mu}\geq F_{\mu}$. 
\label{lem:cookierecur}
\end{lemma}

The proof is adapted from \cite[Lemma~3.4]{HolNicTor12}.  

\begin{proof}
Fix $\tilde{F}_{\mu}<F_{\mu}$ and let $\tilde{C}>0$ be such that 
\begin{equation}
	\label{eq:Frostman bound for recurrence}
\mu(B(x,r))\leq \tilde{C}r^{\tilde{F}_{\mu}}
\end{equation} for all $x\in [0,1]$ and $r>0$. Let $r>0$ and $n\in \N$. We consider $A_r(n)\cap P$ for some $P\in \P_n$.  For simplicity assume $\tilde T^n|_P$ is orientation preserving, otherwise the proof is similar.  Since $\tilde T^n$ is expanding on $P$ the map $x\to \tilde T^{n}x-x$ is strictly increasing on $P$. Thus, if a solution exists there exists a unique $x_{P}^{+}\in P$ such that $\tilde T^nx_{P}^+ = x_{P}^{+}+ r$, and similarly, if a solution exists there exists a unique $x_{P}^{-}\in P$ such that $\tilde T^nx_{P}^- = x_{P}^{-}- r.$ If no solution to $\tilde T^nx = x+ r$ exists then we define $x_{P}^{+}\in P$ to be such that $\tilde T^nx_{P}^{+}-x_{P}^{+}=\max_{x\in P}\tilde T^nx-x.$  If no solution to $\tilde T^nx = x- r$ exists then we define $x_{P}^{-}\in P$ to be such that $\tilde T^nx_{P}^{-}-x_{P}^{-}=\min_{x\in P}\tilde T^nx-x.$ Since there always exists $x\in P$ satisfying $\tilde T^nx-x=0,$ it follows that we always have $\tilde T^nx_{P}^{+}-x_{P}^{+}\leq r$ and $\tilde T^nx_{P}^{-}-x_{P}^{-}\geq -r$. Thus in all cases we have $\tilde{T}^{n}(x_{P}^{-},x_{P}^{+})\subseteq (x_{P}^{-}-r,x_{P}^{+}+r)$. Then due to the monotonicity of $x\to \tilde T^nx-x$ on $P,$ if $x\in P$ satisfies $x<x_{P}^-$ then $x-\tilde T^n(x)>r,$ and if $x>x_{P}^+$ then $\tilde T^n(x)-x>r$.  Hence  $A_r(n)\cap P = (x_{P}^-, x_{P}^+)$. So we wish to estimate $\mu(x_{P}^-, x_{P}^+)$.   Firstly we estimate the length of $\tilde T^n(x_{P}^-, x_{P}^+)$. Using that $x\mapsto \tilde T^n(x)-x$ is expanding by at least $\lambda^n$ and the inclusion $\tilde{T}^{n}(x_{P}^{-},x_{P}^{+})\subseteq (x_{P}^{-}-r,x_{P}^{+}+r),$ it can be shown that $|(x_{P}^-, x_{P}^+)|\le 2r/(\lambda^n-1)$. Using this inequality and the inclusion $\tilde{T}^n(x_{P}^-, x_{P}^+)\subset (x_{P}^--r, x_{P}^++r)$ again, we have the estimate
\begin{align*}
|\tilde T^n(x_{P}^-, x_{P}^+)|& \le |(x_{P}^--r, x_{P}^++r)| \\
&= |(x_{P}^--r, x_{P}^-)|+ |(x^-,x^+)|+ |(x^+,  x^++r)|\\
&\le 2r +\frac{2r}{\lambda^{n}-1}\leq \frac{2\lambda r}{\lambda-1}.
\end{align*}

Let $x_P\in P$ be the unique point such that $\tilde T^n(x_P)=x_P$. It is clear that $x_{P}\in (x_{P}^-,x_{P}^+)$. Using this fact, the bound for $|\tilde T^n(x_{P}^-, x_{P}^+)|$ above, and Lemma~\ref{lem:psireg}, the following holds  $$|S_n\psi(x_P)-S_n\psi(x)|\le K\left(\frac{2\lambda r}{\lambda-1}\right)^\alpha$$ whenever $x\in (x_{P}^-,x_{P}^+)$.  So by $\psi$-conformality,
\begin{align*}
\mu(x_{P}^--r, x_{P}^++r) & \geq  \mu\left(\tilde T^n(x_{P}^-, x_{P}^+)\right) = \int_{(x_{P}^-, x_{P}^+)}\e^{-S_n\psi}~d\mu\\
&\ge  \e^{-K\left(\frac{2\lambda r}{\lambda-1}\right)^\alpha} \e^{-S_n\psi(x_P)}\mu(x_{P}^-, x_{P}^+),\end{align*}
Manipulating the above and using \eqref{eq:Frostman bound for recurrence} implies
\begin{align}
	\label{eq:See Frostman bound}
\mu(x_{P}^-, x_{P}^+)&  \le \e^{K\left(\frac{2\lambda r}{\lambda-1}\right)^\alpha}\e^{S_n\psi(x_P)}\left(\mu(x_{P}^-, x_{P}^+)+ \mu(x_{P}^--r, x_{P}^-) + \mu(x_{P}^+, x_{P}^++r)\right)\nonumber\\
&
\le \e^{K\left(\frac{2\lambda r}{\lambda-1}\right)^\alpha}\e^{S_n\psi(x_P)}\left(\mu(x_{P}^-, x_{P}^+)+ 2\tilde{C}r^{\tilde{F}_\mu}\right).
\end{align}
Since we know that $\psi\le -c_\phi<0$, there exists $\eta>0$ such that for $r$ small enough 
\begin{equation}
	\label{eq:eta space}
	\e^{K\left(\frac{2\lambda r}{\lambda-1}\right)^\alpha}\e^{S_n\psi(x_P)}\le \e^{K\left(\frac{2\lambda r}{\lambda-1}\right)^\alpha}\e^{-c_\phi}<(1-\eta).
\end{equation}  Combining \eqref{eq:See Frostman bound}, \eqref{eq:eta space} and the trivial bound $\e^{K\left(\frac{2\lambda r}{\lambda-1}\right)^\alpha}\leq \e^{K}$ which holds for all $r$ sufficiently small, we have the bound
\begin{equation}
	\label{eq:before sum bound}
 \mu(x_{P}^-, x_{P}^+) \le \frac{2\tilde{C}\e^{K}e^{S_n\psi(x_P)}r^{\tilde{F}_\mu}}{1-\e^{K\left(\frac{2\lambda r}{\lambda-1}\right)^\alpha}\e^{S_n\psi(x_P)}}\le \frac{2\tilde{C}\e^{K}}{\eta} \e^{S_n\psi(x_P)}r^{\tilde{F}_\mu}
 \end{equation} whenever $r$ is sufficiently small. 
Since $\psi$-conformality also implies that there exists $y_P\in P$ such that $\e^{S_n\psi(y_P)}=\mu(P),$ and this $y_{P}$ satisfies $|S_n\psi(x_P)-S_n\psi(y_P)|\le K$ by Lemma \ref{lem:psireg}, \eqref{eq:before sum bound} implies $$ \mu(x_{P}^-, x_{P}^+) \le \frac{2\tilde{C}\e^{2K}\mu(P)}{\eta} r^{\tilde{F}_\mu}.$$
Using this expression and summing over $P\in \P_n$ now yields $$\mu(A_r(n))=\sum_{P\in \P_{n}}\mu(x_{P}^-,x_{P}^+) \le \left(\frac{2\tilde{C}\e^{2K}}{\eta}\right)r^{\tilde{F}_\mu}$$ for all $r$ sufficiently small. Thus our result holds for $r$ sufficiently small. This implies our result for an arbitrary choice of $r>0$ after choosing a potentially larger constant. 
\end{proof}

\begin{lemma}
$\mu$ satisfies the continuous mean scaling property.
\label{lem:ctsscale}
\end{lemma}

\begin{proof} We begin by proving a claim on how cylinders scale.
\begin{claim*}
There exists $\Phi>1$ such that if $P\in \P_n$ and $P'\in \P_{n-1}$ with $P\subset P'$, then $\mu(P') \ge \Phi\mu(P)$.
\end{claim*}

\begin{proof}
There must be some $Q\in \P_n$ where $Q\subset P'$ and  $Q\neq P$.  By conformality,
\begin{align*}
\frac1{\mu(\tilde T^{n-1}Q)}& = \frac{\mu(\tilde T^{n-1}P')}{\mu(\tilde T^{n-1}Q)}= \frac{\int_{P'}\e^{-S_{n-1}\psi}~d\mu}{\int_{Q}\e^{-S_{n-1}\psi}~d\mu}  = \frac{\e^{-S_{n-1}\psi(x)}\mu(P')}{\e^{-S_{n-1}\psi(y)}\mu(Q)}
\end{align*}
for some $x\in P'$ and $y\in Q$.  By Lemma \ref{lem:psireg} we know that $|S_{n-1}\psi(x)-S_{n-1}\psi(y)|\le K.$ This implies
$$\mu(Q) \ge \e^{-K}\mu(\tilde T^{n-1}Q)\mu(P') \ge  \e^{-K}M\mu(P')$$
for $M:=\inf_{\hat Q\in \P_1} \mu(\hat Q)\in (0, 1)$, where we are using the fact that $\tilde T^{n-1}Q\in \P_1$.
So as 
$$\mu(P') \ge \mu(Q) +\mu(P) \ge \e^{-K}M\mu(P')+ \mu(P),$$
 setting $\Phi: = \frac1{1-\e^{-K}M}$ completes the claim.
\end{proof}

Let $x\in \textrm{supp}(\mu)$ and $r>0$. Suppose that $n$ is maximal such that $B(x, r)\cap X$ is contained in some $P\in \P_n$.  Here we also assume that $r$ is small enough that $n\ge 2$.  So $\mu(B(x, r))\le \mu(P)$.  By construction, 
$r\ge G'|P|$, since otherwise $B(x, r)\cap X$ is contained in some $(n+1)$-cylinder.  Then for $P'\in \P_{n-1}$ such that $P\subset P'$, defining $D:= 1/(b_TG')$ we have $Dr\ge |P'|$ so $B(x, Dr)$ contains $P'$.
Hence by the claim above
$$\mu(B(x, Dr)) \ge \mu(P')\ge \Phi\mu(P)\ge  \Phi\mu(B(x, r)).$$
Combining the above with Lemma~\ref{lem:backdouble} we see that $\mu$ satisfies the continuous mean scaling property.
\end{proof}

The following gives a very rough Frostman bound, which seems only likely to be optimal in special cases, but it is sufficient for our purposes here.  We first define $$\lambda_{\max}:= \sup_{x\in \tilde X}|D\tilde T(x)|\ge \lambda>1$$
and recall $\Phi$ in the proof of the lemma above.

\begin{lemma}
	\label{lem:Easy frostman bound}
Let $\Phi$ be as in the claim in Lemma \ref{lem:ctsscale}. The Frostman dimension of $\mu$ satisfies $F_{\mu}\ge \frac{\log\Phi}{\log\lambda_{\max}}$. 
\end{lemma}
\begin{proof}
It is sufficient to prove that $\mu(P) \le |P|^{\frac{\log\Phi}{\log\lambda_{\max}}}$ where $P$ is a cylinder, since a suitable bound then passes to any interval via an approximation argument. If $P\in \P_n$ then $\mu(P)\le \Phi^{-n}$ by repeatedly applying the claim in Lemma \ref{lem:ctsscale}, and $|P| \ge \lambda_{\max}^{-n}$ by an application of the mean value theorem and the chain rule.  Hence  $\mu(P) \le |P|^{\frac{\log\Phi}{\log\lambda_{\max}}}$.
\end{proof}

  Hence we can apply Theorem~\ref{mainthm:1orb} to these systems.  Note that by Lemma~\ref{lem:cookierecur},  the condition  $\beta(C_{\mu}-D_{\mu})<1$ is is implied by $\beta(C_{\mu}-F_{\mu})<1$.
  
  We have also shown that Theorem~\ref{mainthm:2orb} holds when our systems are both the same.  In the other case, we have not proved the existence of $C_{\mu_1, \mu_2}$, see \cite[Section 5]{ArrWat85} to see the types of issues two different measures can cause.  Nevertheless, if we have some estimates in these cases, we may apply Theorem~\ref{thm:2orbgen}. 
 
 \subsection{Cookie cutters: geometric potentials} 
 
 A particularly natural class of equilibrium state for a cookie cutter corresponds to the case when the potential is given by $\phi(x) = \phi_\delta= -\delta\log |D\tilde T(x)|$ for $\delta\in \R$ and $x\in \tilde X$.  Note that this map is always H\"older.  Bowen's formula, see \cite{Bow75} implies that there is a unique $h>0$ such that $P(\phi_{h})=0$ and moreover the Hausdorff dimension of $X$ equals $h$.  Setting $\mu = \mu_{\phi_{h}}$ we see that as in Section~\ref{ssec:cookiegeneral} $\mu$ is $\psi$-conformal where $\psi(x)=  -h\log |D\tilde T(x)|+\log \rho(x)-\log \rho\circ \tilde T(x)$ where $\rho$ is the H\"older density function for $\mu$.   So for example, if $T^n:A\to T^n(A)$ is bijective, then
$$\mu(T^n(A)) = \int_A|D\tilde T^n|^h\cdot \frac{\rho\circ \tilde T^n}{\rho}~d\mu.$$
 It can therefore be shown that in this case $\mu$ is $h$-Ahlfors regular.  Hence $C_{\mu}$ exists and equals $h$. Moreover, by $h$-Ahlfors regularity and Lemma \ref{lem:cookierecur} we must have $F_{\mu}=h$ and $D_{\mu}\geq h$. Consequently the only constraint in applying Theorems~\ref{mainthm:2orb} and \ref{mainthm:1orb} to this system is that $\beta\in(0,2/h)$. This is the optimal parameter space for $\beta$.
 
\subsection{Cookie cutters: an explicit example}
Here we will give a simple example  of a linear cookie cutter with a Bernoulli measure exhibiting multifractal behaviour, where we can compute our quantities explicitly.

Let $I_1= [0, 1/4]$ and $I_2= [3/4, 1]$ and $\tilde T(x) = 4x$ for $x\in I_1$ and $\tilde T(x) = 4x-3$ for $x\in I_2$.  Then $\tilde T$ is a linear cookie cutter.  Let $\mu$ be the corresponding $(\alpha, 1-\alpha)$-Bernoulli measure for $\alpha\in (0, 1/2]$.  We first note that by standard arguments, for example an analogue of \cite[Lemma 13]{GouRouSta24} or an adaptation of the conditioning argument in Section~\ref{ssec:conditioning}, we have $$\underline{C}_\mu= \liminf_{n\to \infty} \frac{\log \sum_{Q\in \mathcal{Q}_n} \mu(Q)^2}{-n\log 4},$$ where $\mathcal{Q}_n$ is the set of intervals of the form $[\frac{k}{4^{n}},\frac{k+1}{4^{n}}]$ for some $k\in \mathbb{N}$. We have a similar expression for $\overline{C}_\mu$ involving the  $\limsup$.  Using these expressions we can show that in this case $C_\mu$ exists and that it equals $-\log\left( \alpha^2+ (1-\alpha)^2\right)/\log 4$. For the Frostman exponent we can compute $F_\mu= -\log(1-\alpha)/\log 4$.

Suppose now that $\alpha=1/3$. Then $C_\mu= \frac{\log\left(\frac95\right)}{\log 4}$  and  $F_\mu=  \frac{\log\frac32}{\log 4}$. Then by Lemma~\ref{lem:cookierecur}, $C_\mu-D_\mu \leq C_\mu-F_\mu=  \frac{\log\frac{18}{15}}{\log 4}\approx 0.132.$ So the conditions $\beta(C_\mu-D_\mu)<1$ and $\beta(C_\mu-F_\mu)<1$ become $\beta<7.604...$, and $\beta C_\mu<2$ becomes $\beta<\frac {2\log 4}{\log \frac95}\approx 3.538$.  Hence the conclusion of Theorem~\ref{mainthm:1orb} holds in the optimal range of parameters $\beta\in (0,\frac{2}{C_{\mu}})$.  We note that Theorem~~\ref{mainthm:2orb} also applies for $\tilde T_1=\tilde T_2=\tilde T$ and $\mu_1=\mu_2=\mu$, and Theorem~\ref{thm:2orbgen} may also apply if $\mu_1$ and $\mu_2$ have different $\alpha$ corresponding to them.

\subsection{A counterexample}
\label{subsec:counterexample}
Consider the version of the Manneville-Pomeau map given by 
\[T(x) = \begin{cases} x(1+2^\alpha x^\alpha) & \text{ if } x\in [0, 1/2),\\
2x-1 &\text{ if } x\in [1/2, 1].
\end{cases}
\]
with parameter $\alpha\in (0,1)$. Then let $\mu$ be the associated acip.  It is shown in \cite[Proposition 3.2]{RouTod24} that  $C_\mu=2(1-\alpha)$ when $\alpha\in (1/2, 1)$: indeed there exists $K>0$ such that for all $r\in(0,1)$ we have
\begin{equation}
	\label{eq:MP correlation dimension}
	\frac{r^{2(1-\alpha)}}K\le \int\mu\left(B(z, r)\right)~d\mu(z)\le Kr^{2(1-\alpha)}.
	\end{equation}
The continuous mean scaling property can be easily checked here, and it can be shown that the Frostman dimension satisfies $F_\mu=1-\alpha$, which would not cause problems in terms of the statement of Theorem~\ref{mainthm:1orb} for $\beta\in (0, 2]$; however, as in \cite[Corollary 1]{Sar02}, the mixing rate here is polynomial, so the conditions for our main results do not hold.  We will consider the particular case of $\beta=1$ and show that the conclusion of Theorem \ref{mainthm:1orb} does not hold.

As in, for example,  
\cite[Corollary 1]{Sar02}, let $y_0=1$, and for $n\geq 1$ let $y_n\in (0, y_{n-1})$ be such that $T(y_n)= y_{n-1}$.  As proved there, $y_n\sim c/n^{1/\alpha}$.  Let $k_n:= \lfloor (n\log n)^{\alpha}\rfloor $ so that $y_{k_n}  \sim \frac c{n\log n}$.  Let $\gamma\in (0, \alpha)$ and observe that 
$$k_n- n^\alpha (\log n)^\gamma = k_n\left(1- \frac{n^\alpha (\log n)^\gamma}{\lfloor (n\log n)^{\alpha}\rfloor}\right)>\frac{k_{n}}{2}$$ for $n$ sufficiently large. This inequality and the fact $y_n\sim c/n^{1/\alpha}$ mean that for any $s>0,$ for $n$ sufficiently large, if $T^nx\in (0, y_{k_n})$ then $T^{n+i}x\in (0, y_{k_n-i})\in (0, s/2n)$ for all $0\le i\le n^\alpha (\log n)^\gamma$.  Therefore for all $s>0,$ for $n$ sufficiently large, if $T^nx\in (0, y_{k_n})$ then
 $$\#\left\{0\leq i\neq j\leq n+n^{\alpha}(\log n)^{\gamma}: |T^{i}x-T^{j}x|<\frac{s}{2n}\right\}\gtrsim n^{2\alpha} (\log n)^{2\gamma}$$
 Using this counting bound and \eqref{eq:MP correlation dimension}, we see that for any $s>0$, for $n$ sufficiently large, if $T^{n}x\in (0,y_{k_{n}})$ then
 \begin{align}
 &\frac{ \#\left\{0\le i\neq  j<2n:|T^ix- T^jx|\le \frac s{2n}\right\}}{(2n)^2\int\mu\left(B\left(z,\frac s{2n}\right)\right)~d\mu(z)} \nonumber \\
 &\hspace{2cm} \asymp \frac1{n^{2\alpha}} \#\left\{0\le i\neq  j<2n:|T^ix- T^jx|\le \frac s{2n}\right\} \gtrsim (\log n)^{2\gamma}.
 \label{eq:blowup}
 \end{align}

 Since the Lebesgue measure of $(0, y_{k_n})$ is at least a constant multiple of $\frac1{n\log n}$ for all $n\in \N$, hence not summable, by  \cite[Theorem 1.1]{Gou07} for $\mu$ almost every $x$, there are infinitely many $n$ such that $T^nx\in (0, y_{k_n})$. Hence for any $s>0$ there exists infinitely many $n$ so that \eqref{eq:blowup} is satisfied. Hence for $\beta=1,$ for $\mu$ almost every $x$ the conclusion of Theorem \ref{mainthm:1orb} does not hold for any $s>0$. 
 
\begin{remark}
A natural question raised by this counterexample is, for a system with some polynomial decay of correlations: for sufficiently small $\beta$ could one establish results as in Theorems~\ref{mainthm:2orb} and \ref{mainthm:1orb}?  
The methods to prove our main theorems suggest that we could make such an extension, though given the known examples, one would expect this problem to come with less suitable Banach spaces, which may then incur additional error terms due to the approximation of indicator functions on intervals by elements of those Banach spaces.  Moreover, one would have to find replacements for our exponential 4-mixing for intervals property and our early return property.
Therefore to extend our results much beyond the examples covered here, it is likely that a suite of new techniques would be required.
\end{remark}


\begin{thebibliography}{XXX}

\bibitem[AD]{AarDen01} J.\ Aaronson, M.\ Denker, \emph{Local limit theorems for partial sums of stationary sequences generated by Gibbs-Markov maps,} Stoch. Dyn. \textbf{1} (2001) 193--237. 

\bibitem[AB]{AisBak} C.\ Aistleitner, S.\ Baker, \emph{On the pair correlations of powers of real numbers}, Israel J. Math. {\bf 242} (2021) no.~1, 243--268.

\bibitem[ALP]{AisLacPau18} C.\ Aistleitner, T.\ Lachmann, F.\ Pausinger, \emph{Pair correlations and equidistribution,} J. Number Theory {\bf 182} (2018) 206--220.


\bibitem[ALL]{AisLarLew17}  C.\ Aistleitner, G.\ Larcher, M.\ Lewko, 
\emph{Additive energy and the Hausdorff dimension of the exceptional set in metric pair correlation problems,}
 Israel J.\ Math.\ \textbf{222} (2017) 463--485, With an appendix by Jean Bourgain. 
 
\bibitem[ALT]{AisLarTec19}  C.\ Aistleitner, G.\ Larcher, N.\ Technau, \emph{There is no Khintchine threshold for metric pair correlations,} Mathematika {\bf 65} (2019)  929--949.

 \bibitem[ABB]{AllBakBar25} D.\ Allen, S.\ Baker, B.\ B\'{a}r\'{a}ny, \emph{Recurrence rates for shifts of finite type,} Adv. Math. \textbf{460} (2025), Paper No. 110039, 36 pp.
 
 \bibitem[ABM]{AlvBecMer23} N.\ Alvarez, V.\ Becher, M.\ Mereb, \emph{Poisson generic sequences,} Int. Math. Res. Not. IMRN (2023) 20970--20987.

 \bibitem[AW]{ArrWat85} R.\ Arratia, M.S.\ Waterman, \emph{Critical phenomena in sequence matching,} Ann. Probab.
\textbf{13} (1985) 1236–1249.

\bibitem[BF]{BakFar21} S.\ Baker, M.\ Farmer, \emph{Quantitative recurrence properties for self-conformal sets,} Proc. Amer. Math. Soc. {\bf 149} (2021) 1127--1138.

\bibitem[BK]{BakKoi24} S.\ Baker, H. Koivusalo, \emph{Quantitative recurrence and the shrinking target problem for overlapping iterated function systems,} Adv. Math. {\bf 442} (2024), Paper No. 109538, 65 pp.

\bibitem[BLR]{BarLiaRou19} V.\ Barros, L.\ Liao, J.\ Rousseau,
  \emph{On the shortest distance between orbits and the longest
    common substring problem,} Adv. Math. \textbf{344} (2019)
  311--339.
  
  \bibitem[BCGW]{BloChoGalWal18} T.F.\ Bloom, S.\ Chow, A.\ Gafni, A.\  Walker, 
\emph{Additive energy and the metric Poissonian property,}
Mathematika \textbf{64} (2018) 679--700.

\bibitem[B]{Bos93} M.D.\ Boshernitzan, \emph{Quantitative recurrence results,} Invent. Math. \textbf{113} (1993)  617--631.

\bibitem[Bow]{Bow75} R.\ Bowen, \emph{Equilibrium States and the Ergodic Theory of Anosov Diffeomorphisms,} Springer Lect. Notes in Math. 470 (1975).

\bibitem[BDT]{BruDemTod18} H.\ Bruin, M.\ Demers, M.\ Todd, 
\emph{Hitting and escaping statistics: mixing, targets and holes, } Adv. Math. \textbf{328} (2018) 1263--1298.


\bibitem[B]{Bugeaudbook} Y.\ Bugeaud, \emph{ Distribution modulo one and Diophantine approximation}, Cambridge Tracts in Mathematics, 193, Cambridge Univ. Press, Cambridge, 2012.

\bibitem[EBMV]{BazMarVin} D. El-Baz, J. Marklof, I. Vinogradov, \textit{The two-point correlation function of the fractional parts of $\sqrt{n}$ is Poisson,} Proc. Amer. Math. Soc. {\bf 143} (2015), no.~7, 2815--2828.

\bibitem[EM]{ElkMcM04} N.\ Elkies, C.\ McMullen, \emph{Gaps in ${\sqrt n}\bmod 1$ and ergodic theory,} Duke Math. J. {\bf 123} (2004) 95--139.   
  
\bibitem[FFT1]{FreFreTod10} A.C.M.\ Freitas, J.M.\ Freitas, M.\ Todd, \emph{Hitting time statistics and
extreme value theory,} Probab. Theory Related Fields \textbf{147} (2010) 675--710.  

\bibitem[FFT2]{FreFreTod25} A.C.M.\ Freitas, J.M.\ Freitas, M.\ Todd, \emph{ Enriched functional limit
theorems for dynamical systems,}  Online at Ann. Sc. Norm. Super. Pisa Cl. Sci. (5).
  
\bibitem[G]{Gou07} S.\ Gou\"ezel, \emph{A Borel-Cantelli lemma for intermittent interval maps,}
 Nonlinearity \textbf{20} (2007) 1491--1497. 
 
 \bibitem[GRS]{GouRouSta24} S.\ Gou\"ezel, J.\ Rousseau, M.\ Stadlbauer, \emph{Minimal distance between random orbits,} Probab. Theory Related Fields \textbf{189} (2024) 811--847.
 
 \bibitem[GL]{GreLar17} S.\ Grepstad, G.\ Larcher, \emph{On pair correlation and discrepancy,} Arch. Math. (Basel) {\bf 109} (2017), no.~2, 143--149.
   
 
  \bibitem[HKKP]{HolKirKunPer24} M.\ Holland, M.\ Kirsebom, P.\
  Kunde, T.\ Persson, \emph{Dichotomy results for eventually
    always hitting time statistics and almost sure growth of
    extremes,} Trans. Amer. Math. Soc. \textbf{377} (2024) 3927--3982.

  
  \bibitem[HNT]{HolNicTor12} M. Holland, M. Nicol, A. Török, \emph{Extreme value
    theory for non-uniformly expanding dynamical systems},
  Trans. Amer. Math. Soc. \textbf{364} (2012) 661--688.

\bibitem[HT]{HolTod25} M.  Holland, M.\ Todd, \emph{On distributional limit laws for recurrence,}
 Nonlinearity \textbf{38} (2025) Paper No. 075028.

\bibitem[KMS]{KesMunStr12} M.\ Kesseb\"ohmer, S.\ Munday,  B.O.\ Stratmann, 
\emph{Strong renewal theorems and {L}yapunov spectra for $\alpha$-Farey and $\alpha$-L\"uroth systems,} Ergodic Theory Dynam. Systems \textbf{32} (2012) 989--1017.

\bibitem[KKP]{KirKunPer23} M.\ Kirsebom, P.\ Kunde, T.\ Persson,
  \emph{On shrinking targets and self-returning points, }
  Ann. Sc. Norm. Super. Pisa Cl. Sci. (5) \textbf{24} (2023)
  1499--1535.

\bibitem[KKPT]{KirKunPerTod25} M.\ Kirsebom, P.\ Kunde, T.\ Persson, M.\ Todd,
\emph{Almost sure orbits closeness}, Preprint (arXiv:2510.13277).


\bibitem[KP]{KuiNie} L.\ Kuipers, H.\ Niederreiter, \emph{Uniform distribution of sequences}, Pure and Applied Mathematics, Wiley-Intersci., New York-London-Sydney, 1974.

\bibitem[LLSV]{LeLiSiVe25} J.\ Levesley, B.\ Li, D.\ Simmons, S. Velani, \emph{Shrinking targets versus recurrence: the quantitative theory}, Mathematika {\bf 71} (2025) Paper No. e70039, 16 pp.

\bibitem[LST]{LutSouTec25} C.\ Lutsko, A.\ Sourmelidis, N.\ Technau, \emph{Pair correlation of the fractional parts of $\alpha n^{\theta}$,} J. Eur. Math. Soc. (JEMS) {\bf 27} (2025) 4069--4082.

\bibitem[M]{Mar20} J. Marklof, \emph{Pair correlation and equidistribution on manifolds,} Monatsh. Math. {\bf 191} (2020) 279--294.

\bibitem[MN]{MelNic08} I.\ Melbourne,  M.\ Nicol, \emph{Large deviations for nonuniformly hyperbolic systems, }Trans. Amer. Math. Soc. \textbf{360} (2008) 6661--6676. 

\bibitem[MZ]{MelZwe15} I.\ Melbourne, R.\ Zweim\"uller, \emph{Weak convergence to stable {L}\'evy processes for nonuniformly hyperbolic dynamical systems,} Ann. Inst. Henri Poincar\'e Probab. Stat. \textbf{51} (2015) 545--556.

\bibitem[NP]{NaiPol07} R.\ Nair, M.\ Pollicott, \emph{Pair correlations of sequences in higher dimensions,}
 Israel J. Math. \textbf{157} (2007) 219--238.
 

\bibitem[PW]{PesWei97} Y.\ Pesin, H.\ Weiss, \emph{A multifractal analysis of equilibrium measures for conformal expanding maps and {M}oran-like geometric constructions,}
 J. Statist. Phys. \textbf{86} (1997) 233--275.
 
\bibitem[PU]{PrzUrb10} F.\ Przytycki, M.\ Urba\'nski, Conformal fractals: ergodic theory methods, London Mathematical Society Lecture Note Series, vol. 371, Cambridge University Press, Cambridge, 2010.

\bibitem[Ra]{Ran89} D.A.\ Rand, \emph{The singularity spectrum $f(\alpha)$ for cookie-cutters,} Ergodic Theory Dynam. Systems \textbf{9} (1989) 527--541.

\bibitem[Re]{Reg23} S.\ Regavim, \emph{Minimal gaps and additive energy in real-valued sequences,}
 Q. J. Math. \textbf{74} (2023)  825--866.

\bibitem[RSW]{ReScWa05}  G.\ Reinert, S.\ Schbath, M.\ Waterman, \emph{Applied Combinatorics on Words}, volume 105 of Encyclopedia of Mathematics and its Applications, chapter Statistics on Words with Applications to Biological
Sequences, Cambridge University Press, 2005.

\bibitem[RY]{ReyYou08} L.\ Rey-Bellet,  L.-S.\ Young, \emph{Large deviations in non-uniformly hyperbolic dynamical systems,} Ergodic Theory Dynam. Systems \textbf{28} (2008) 587--612.

\bibitem[RT]{RouTod24} J.\ Rousseau, M.\ Todd, 
\emph{Orbits closeness for slowly mixing dynamical systems,}
 Ergodic Theory Dynam. Systems \textbf{44} (2024) 1192--1208.

 \bibitem[Ru]{Rud18} Z.\ Rudnick, \emph{A metric theory of minimal gaps, }
 Mathematika \textbf{64} (2018) 628--636.

\bibitem[RS]{RudSar} Z.\ Rudnick, P.\ Sarnak, \emph{The pair correlation function of fractional parts of polynomials,} Comm. Math. Phys. {\bf 194} (1998) 61--70.

\bibitem[RZ1]{RudZah99} Z.\ Rudnick, A.\ Zaharescu, \emph{A metric result on the pair correlation of fractional parts of sequences,}
Acta Arith. \textbf{89} (1999) 283--293. 



\bibitem[RZ2]{RudZah02} Z.\ Rudnick, A.\ Zaharescu, 
\emph{The distribution of spacings between fractional parts of
lacunary sequences,} Forum Math. \textbf{14} (2002) 691--712.

\bibitem[Ry]{Ryc83} M.\ Rychlik, \emph{Bounded variation and invariant measures,}
 Studia Math. \textbf{76} (1983) 69--80.
 
 
\bibitem[S]{Sar02} O.\ Sarig, \emph{Subexponential decay of correlations,} Invent. Math. \textbf{150}
 (2002) 629--653.

\bibitem[Sau]{Sau09} B.\ Saussol, \emph{An introduction to quantitative Poincar\'{e} recurrence in dynamical systems,} Rev. Math. Phys. \textbf{21} (2009) 949--979.


\bibitem[W]{Wal} A.\ Walker, \emph{The primes are not metric Poissonian}, Mathematika {\bf 64} (2018) 230--236.

\bibitem[Wal]{Waltersbook} P.\ Walters, \emph{An introduction to ergodic theory}, Graduate Texts in Mathematics, 79, Springer, New York-Berlin, 1982 
 
\bibitem[Wat]{Wat95} M.\ Waterman, \emph{Introduction to Computational Biology: Maps, Sequences and Genomes}, Chapman and Hall, London, 1995

\bibitem[Wei]{Wei20} B.\ Weiss, \emph{ Poisson generic points,} 23-27 November 2020. Jean-Morlet Chair conference on Diophantine Problems, Determinism and Randomness. Centre International
de Rencontres Math\'ematiques. Audio-visual resource: doi:10.24350/CIRM.V.19690103.

\bibitem[Y]{You98} L.-S.\ Young, \emph{Statistical properties of dynamical systems with some hyperbolicity,} Ann. of Math. (2) \textbf{147} (1998) 585--650.

\bibitem[Z]{Zha24} B.\ Zhao,
\emph{ Closest distance between iterates of typical points,}
Discrete Contin. Dyn. Syst. \textbf{44} (2024) 2252--2279.

 \end{thebibliography}
 \end{document}